\documentclass[a4, 10pt,leqno]{amsart}
\usepackage{lmodern}
\usepackage[english]{babel}
\usepackage{microtype}
\usepackage{mathtools}
\usepackage{bbm}
\usepackage[utf8]{inputenc}
\usepackage[T1]{fontenc}
\usepackage{upgreek}
\usepackage{listings}
\DeclareMathAlphabet{\mathdutchcal}{U}{dutchcal}{m}{n}
\usepackage{amsmath,amssymb,amscd,amsthm,graphicx,enumerate}
\usepackage{mathrsfs}
\usepackage{amssymb,amsthm,a4wide}
\usepackage{nicefrac}
\newcounter{example}
\newenvironment{example}[1]{\refstepcounter{example}\par\medskip
	\noindent \textsc{\small Example~\theexample. #1} \rmfamily\hspace{-2pt}}{\medskip}
\usepackage{pifont}

\usepackage{color}
\usepackage[svgnames]{xcolor}
\usepackage{xcolor}
\definecolor{ultrablue}{rgb}{0.1,0.1,1}
\usepackage{hyperref}
\hypersetup{
	linktoc=all,
	colorlinks,
	linkcolor={ultrablue},
	citecolor={ultrablue},
	urlcolor={ultrablue}
}
\hypersetup{breaklinks=true}
\usepackage{cite}

\allowdisplaybreaks
\makeatletter 
\def\l@subsection{\@tocline{2}{0pt}{3pc}{6pc}{}}
\makeatother

\makeatletter
\@namedef{subjclassname@2020}{%
	\textup{2020} Mathematics Subject Classification}
\makeatother
\makeatletter 
\def\l@subsection{\@tocline{2}{0pt}{3pc}{6pc}{}}
 \makeatother
\usepackage[a4paper,bindingoffset=0in,%
left=1.2in,right=1.2in,top=1.2in,bottom=1.2in,
]
{geometry}
\theoremstyle{plain}
\newtheorem{lemma}{Lemma}[section]
\newtheorem{theorem}[lemma]{Theorem}

\theoremstyle{definition}
\newtheorem{definition}[lemma]{Definition}
\newtheorem{remark}[lemma]{Remark}

\newcommand{\introthmname}{}
\newtheorem{introthminn}{\introthmname}
\newenvironment{introthm}[1]
{\renewcommand{\introthmname}{#1}\begin{introthminn}}
	{\end{introthminn}}
\newcommand{\introcorname}{}

\numberwithin{equation}{section}

\newcommand{\fs}{\mathsf{C}_{\sf fs}}
\newcommand{\cb}{\mathsf{C}_{\sf b}}
\newcommand{\R}{\mathbb R}
\DeclareMathOperator{\Ric}{Ric}

\newcommand{\Was}{\mathcal{W}}
\newcommand{\m}{\mathdutchcal{m}}
\newcommand{\G}{\mathcal{G}}
\newcommand{\B}{\mathsf{B}}

\newcommand{\T}{\mathrm{T}}

\newcommand{\LL}{\mathsf{L}}
\newcommand{\Lip}{\mathsf{Lip}}
\newcommand{\dist}{\mathrm{d}}

\newcount\stylenum \newdimen\styledim 
\def\varstyle#1{\mathchoice{\stylenum=0 #1}{\stylenum=1 #1}{\stylenum=2 #1}{\stylenum=3 #1}}
\def\mathaxis{\fontdimen22\ifcase\stylenum 
	\textfont\or\textfont\or\scriptfont\or\scriptscriptfont\fi2 }
\def\setstyledim{\styledim=\ifcase\stylenum .1em\or.1em\or.07em\or.05em\fi\relax}

\def\sqdot{\mathbin{\varstyle{\raise\mathaxis\hbox{\setstyledim
				\kern\styledim 
				\vrule width1.2\styledim height.6\styledim depth.6\styledim
				\kern\styledim}}}}
\newcommand{\ident}{\raisebox{0pt}{\scalebox{1.1}{$\mathbbm{1}$}}\hspace{-1pt}}

\newcommand{\eps}{\varepsilon}
\renewcommand{\epsilon}{\varepsilon}
\newcommand{\restr}{\raisebox{-.1908ex}{$\big|$}}
\newcommand{\Lo}{\mathcal{L}}
\newcommand{\Dom}{\sf Dom}
\DeclareMathOperator{\supp}{supp}
\newcommand\dunderline[3][-1pt]{{%
		\sbox0{#3}%
		\ooalign{\copy0\cr\rule[\dimexpr#1-#2\relax]{\wd0}{#2}}}}
\newcommand\smallO{
	\mathchoice
	{{\scriptstyle\mathcal{O}}}
	{{\scriptstyle\mathcal{O}}}
	{{\scriptscriptstyle\mathcal{O}}}
	{\scalebox{.7}{$\scriptscriptstyle\mathcal{O}$}}
}
\usepackage{array}%
\usepackage{esvect}
\usepackage[super]{nth}
\begin{document}
\title[\tiny Ricci curvature of semigroup walks and generator representations]{\small Ricci curvature of semigroup walks on graphs\\ and generator representations} 
\address{Sajjad Lakzian \\ \newline \indent Department of Mathematical Sciences\\ \newline \indent Isfahan University of Technology (IUT) \\  \newline \indent Isfahan 8415683111, Iran}
\email{\href{mailto:slakzian@iut.ac.ir}{slakzian@iut.ac.ir}}
\maketitle
\vspace{-14pt}
\begin{center}
\begin{minipage}[t]{0.5\textwidth}
	\centering
	\textbf{Sajjad Lakzian}
\end{minipage}
\end{center}
\begin{abstract}
		\par \textsl{We present an operator-theoretic framework for the Ollivier--Ricci curvature of (one-parameter families of) random walks generated by positive semigroups on graphs. On a metrically complete graph $(\G,\dist)$ that is $\dist$-locally finite, we introduce a class of admissible semigroup random walks whose infinitesimal first-order approximation is itself a random walk. We prove that the Ollivier--Ricci curvature of an admissible semigroup random walk is well-defined and completely determined by its first-order approximation. We then establish a variational representation of the curvature of random walks in terms of a notion of the initial-time generator. This representation applies in particular to admissible semigroup random walks. Finally, we use this representation to define an operator-theoretic Ollivier--Ricci curvature for operators acting on spaces of functions over general metric spaces and establish sufficient conditions for well-definedness. This framework contains, in particular, Markovian random walks associated with discrete Laplacians.}
\end{abstract}
\date{\today}
\tableofcontents
\section{Introduction}
\par Ricci curvature, one of the central notions in Riemannian geometry, encodes the interaction between the metric structure, geodesic behavior, diffusion, and volume growth. Extending the concept of Ricci curvature beyond smooth manifolds has led to several notions of Ricci curvature (and Ricci curvature lower bounds) for metric and measure spaces. Among these approaches, Lott--Sturm--Villani and Ollivier--Ricci curvature bounds have received particular attention due to their intrinsic connection with optimal transport and the contraction properties of the heat flow.
\par For weighted graphs and/or Markov chains, several notions of Ricci curvature bounds have been developed from different perspectives. These include transport-based curvature bounds inspired by the Lott--Sturm--Villani theory~\cite{BS,Erbar-Maas}, discrete Bakry--\'Emery curvature bounds~\cite{LY,CLP,FL}, Ollivier--Ricci curvature defined through Wasserstein contraction of random walks initiated in~\cite{Ol}, Lin--Lu--Yau curvature (based on Ollivier's approach) for certain classes of lazy random walks~\cite{LLY}, and Forman--Ricci curvature arising from combinatorial Bochner-type decomposition~\cite{For}. Each of these theories captures different geometric features and is suited to different applications.
\par Among these approaches, Ollivier--Ricci curvature is particularly natural from the viewpoint of Markov semigroups. It measures curvature through the infinitesimal contraction rate of the Wasserstein distance between probability measures generated by a random walk. This perspective is closely related to heat-flow characterizations of Ricci curvature and contraction properties of diffusion semigroups as established in~\cite{Ol,MW}.  
\par The Ollivier--Ricci curvature of a one-parameter family of random walks $\upmu_z^\varepsilon$ is given by
\begin{equation*}
	\Ric(x,y) := \text{-} \, \dist(x,y)^{-1}\;\nicefrac{d^+}{d\varepsilon}\restr_{\varepsilon=0} \;\Was_1\left( \upmu_x^{\varepsilon}, \upmu_y^{\varepsilon} \right),\quad \text{for} \quad x\neq y.
\end{equation*}
if the limit (the right derivative at $\eps = 0$) exists. \emph{Hereafter, we will simply refer to the one-parameter family of random walks $\upmu_z^\varepsilon$ as a random walk}.
\par On graphs, the Ricci curvature has been studied mainly for time-affine lazy random walks (see \hyperref[defn:lazy-walk]{Definition~\ref{defn:lazy-walk}}) and random walks generated by the general discrete Laplacian. In the aforementioned setting, a variational or limit-free characterization has been obtained and several similarities with classical Ricci curvature have been established; see e.g.~\cite{MW}. However, much less is known for random walks generated by more general operators. 
\par This leads to the following fundamental questions:
\begin{enumerate}
	\item [] \textbf{\small Question 1:} For which classes of random walks is the Ollivier--Ricci curvature well-defined?
	\smallskip
	\item [] \textbf{\small Question 2:} Can the Ollivier--Ricci curvature be characterized by a variational (limit-free) formulation?
\end{enumerate}
\subsection*{The aim of this paper}
\par The purpose of this paper is to develop an operator-theoretic framework for the Ollivier--Ricci curvature of random walks generated by positive semigroups on graphs thereby addressing the questions posed above for a broad class of semigroup walks generated by semi-local operators. 
\par Key concepts introduced in this paper include: 
\begin{itemize}
	\item Normal forms for general random walks; see \hyperref[subsec:normal-form]{Section~\ref{subsec:normal-form}}.
	\smallskip
	\item Operator-theoretic concepts of quasi-symmetry and semilocality; see \hyperref[subsubsec:sym-loc]{Section~\ref{subsubsec:sym-loc}}.
	\item The conical-supremum construction which is a way to extend the domain of a positive operator in the direction of a cone of functions at the cost of losing the full linearity in the domain and the operator; see \hyperref[subsubsec:conical-sup]{Section~\ref{subsubsec:conical-sup}}.
	\smallskip
	\item The notion of operator-theoretic Ricci curvature; see \hyperref[sec:op-theo-ric]{Section~\ref{sec:op-theo-ric}}.
\end{itemize}
\subsection*{Overview of Main Results}
Consider a triple $\left( \G,\dist\right)$, where $\G$ is a simple graph and $\dist$ is a metric on $\G$ such that $(\G,\dist)$ is complete and hence Polish. We further assume that $\G$ is $\dist$-locally finite, meaning that every metric ball is finite. 
\par We consider an important class of random walks: semigroup walks generated by a suitable $\mathsf C^0$-semigroup of operators satisfying admissibility conditions; see \hyperref[sec:semigroup-walks]{Section~\ref{sec:semigroup-walks}} for the required operator-theoretic preparation and details.
\begin{definition} \label{def:adms-semigrp-walk}
		Suppose $\left(\B, \|\cdot\|\right)$ is a Banach space of functions containing $\mathsf{C}_{\sf b}$. Let $\T_\eps: \B \to \B$ be a $\mathsf{C}^0$-semigroup of operators that is $2$-regular, positive and formally stochastically complete and the domain of the infinitesimal generator $\Lo$ contains $\Lip(\G) \cap \mathsf{C}_{\sf b}$.
	\begin{itemize}
\item \textbf{\small Admissible forward semigroup walk:} Suppose
$
\upmu_z^\epsilon = \T_\eps(\delta_z)
$
is a random walk. It will be called an admissible forward semigroup walk if the following conditions are satisfied:
\begin{enumerate}
	\smallskip
	\item [(F1)]\label{item:admis-walk-b} The \nth{2} order remainder 
	$
	\mathcal{R}^2_\eps := \T_\varepsilon-(\ident+\varepsilon\mathcal{L})
	$
	is uniformly (in $\varepsilon$) roughly quasi-symmetric.
	\smallskip
	\item [(F2)]\label{item:admis-walk-c} The infinitesimal generator $\mathcal{L}$ is strongly semi-local at a controlled scale, comparable to multiplication and bounded below in the sense of \hyperref[defn:bounded-below]{Definition~\ref{defn:bounded-below}}.
\end{enumerate}
\medskip
\item \textbf{\small Admissible backward semigroup walk:} 
Suppose 
$
\upmu_z^\epsilon(w) = \T_\eps(\delta_w)(z)
$
defines a random walk. It will be called an admissible backward semigroup walk if the following conditions are satisfied:
\begin{enumerate}
	\smallskip
	\item [(B1)]\label{item:back-admis-walk-b} The infinitesimal generator $\mathcal{L}$ is uniformly semi-local, comparable to multiplication and bounded below in the sense of \hyperref[defn:bounded-below]{Definition~\ref{defn:bounded-below}}.
\end{enumerate}
\end{itemize}
\end{definition}
\par We note that for the admissible backward semigroup walk, the conditions on the generator are natural in the sense that they are modeled on the properties of the general discrete Laplacian. For admissible forward semigroup walks, the symmetry condition generalizes the symmetry of the normalized Laplacian.
\par Our first main result is that admissible semigroup walks can be approximated to the second order in Wasserstein distance by their $1$-jets 
$
	{}^1\upmu_z^\varepsilon
	:= (\ident + \eps \Lo) \delta_z
$,
and consequently, $\Ric$ is well-defined and the Ollivier--Ricci curvature of the corresponding $1$-jet random walk, ${}^1\!\Ric$. This is established using the key idea of conical-supremum extension, $\widetilde{^\mathsf{S}\T}$, of a positive operator, $\T$; this is a concept that is introduced in \hyperref[subsubsec:conical-sup]{Section~\ref{subsubsec:conical-sup}}. 
\begin{introthm}{\bf Main Theorem}[Well-definedness of $\Ric$]\label{thm:main-1}
Let $\upmu_z^\epsilon = \T_\eps(\delta_z)$ be an admissible forward or backward semigroup walk. Then, for all $z\in \G$, we have 
\begin{equation*}
	\Was_1\left(\upmu_z^\varepsilon,{}^1\upmu_z^\varepsilon\right)
	=
	\mathcal{O}\left(\varepsilon^2\right)
	\qquad\text{as}\qquad
	\varepsilon\downarrow0.
\end{equation*}	
	Consequently, $\Ric$ is well-defined and coincides with ${}^1\!\Ric$.	 
\end{introthm}
\par Our second main result applies to general random walks. It provides a representation of the upper Ollivier--Ricci curvature, $\Ric^+(x,y)$, by a variational formula in terms of a generalized notion of generator for random walks; see \hyperref[subsubsec:ct-olric]{Section~\ref{subsubsec:ct-olric}} for the definition of  $\Ric^+(x,y)$. This variational representation in terms of the generator requires suitable second-order asymptotics and continuity assumptions on the generator. 
\par Let $\upmu_z^\varepsilon$ be a given random walk (not necessarily a semigroup walk) with the averaging operator $\mathscr{A}^\epsilon$, integrated Laplacians $\mathscr{L}^\epsilon$ and backward initial generator $\mathscr L$; see \hyperref[subsec:normal-form]{Section~\ref{subsec:normal-form}} for the definitions of these objects. See \hyperref[eq:derivative]{\eqref{eq:derivative}} for the definition of the directional derivative $\nabla_{\vv{xy}} f$.
\begin{introthm}{\bf Main Theorem}[Variational characterization of $\Ric^+$]\label{thm:main-2}
	Suppose that the following hypotheses hold:
	\begin{enumerate}
		\item [(V1)]\label{item:lff-1} $\mathscr{L}^\epsilon$ (equivalently $\mathscr{A}^\epsilon$) has uniform second-order asymptotics at $\eps = 0$ along $\Lip_1(\G) \cap \fs$ (see \hyperref[subsubsec:unif-asymp]{Section~\ref{subsubsec:unif-asymp}});
		\smallskip
		\item [(V2)]\label{item:lff-2} $\mathscr{L}$ is continuous with respect to the pointwise convergence of finitely supported functions.
	\end{enumerate}
	Then
	\begin{equation*}
		\Ric^+(x,y)
		\le
		\inf_{\substack{f \in \Lip(1)\\ \nabla_{\vv{xy}} f = 1 }}
		\nabla_{\vv{yx}} \mathscr{L}(f).
	\end{equation*}
	If, furthermore, $\mathscr{L}$ is semi-local at a controlled scale and $\Ric^+(x,y)$ is finite, then
	\begin{equation*}
		\Ric^+(x,y)
		=
		\inf_{\substack{ \\ f \in \Lip(1) \\ \nabla_{\vv{xy}} f = 1 }}
		\nabla_{\vv{yx}} \mathscr{L}(f)
		=
		\inf_{\substack{ \\ f \in \Lip(1) \cap \fs\\ \nabla_{\vv{xy}} f = 1 }}
		\nabla_{\vv{yx}} \mathscr{L}(f).
	\end{equation*}
\end{introthm}
It is worth noting that a sufficient condition for \hyperref[item:lff-2]{(V2)} to hold is if $\mathscr{L}$ is continuous w.r.t. the point-open topology.
\par As a main consequence of \hyperref[thm:main-2]{Main Theorem~\ref{thm:main-2}}, for admissible semigroup walks, we obtain a variational representation of $\Ric$ in terms of the infinitesimal generator. 
\begin{introthm}{\bf Main Corollary}[Generator representations of $\Ric$]\label{cor:main-2}
	\hfill
	\begin{enumerate}
		\item \label{item:GR1}
		For an admissible forward semigroup walk $\upmu_z^\epsilon = \T_\eps(\delta_z)$ with infinitesimal generator $\Lo$ for the semigroup $\T_\eps$, we have
		\begin{equation*}
			\Ric(x,y)
			=
			\inf_{\substack{f \in \Lip(1)\cap\fs\\
					\nabla_{\vv{xy}}f=1}}
			\left<f, \Lo \left(\delta_x - \delta_y \right)\right>
			=
			\inf_{\substack{f \in \Lip(1)\cap \cb\\
					\nabla_{\vv{xy}}f=1}}
			\left<f, \Lo \left(\delta_x - \delta_y \right)\right>.
		\end{equation*}
		If $\mathcal{L}$ is further roughly symmetric (see \hyperref[subsubsec:sym-loc]{Section~\ref{subsubsec:sym-loc}}), then
		\begin{align}\label{eq:var-char-neat}
			&\Ric(x,y)\\
			&=
			\inf_{\substack{f \in \Lip(1) \cap \fs\\
					\nabla_{\vv{xy}} f = 1 }}
			\nabla_{\vv{yx}} \mathcal{L}(f)
			=
			\inf_{\substack{f \in \Lip(1)\cap \cb\\
					\nabla_{\vv{xy}} f = 1 }}
			\nabla_{\vv{yx}} \mathcal{L}(f)
			=
			\inf_{\substack{f \in \Lip(1)\cap \Dom(\Lo)\\
					\nabla_{\vv{xy}} f = 1 }}
			\nabla_{\vv{yx}} \mathcal{L}(f).\notag
		\end{align}
		\item \label{item:GR2} For an admissible backward semigroup walk $\upmu_z^\epsilon(w) = \T_\eps(\delta_w)(z)$ with the property that the infinitesimal generator $\Lo$ of $\T_t$ is continuous w.r.t. the pointwise convergence of finite support functions, the generator representation \hyperref[eq:var-char-neat]{\eqref{eq:var-char-neat}} holds true. 
	\end{enumerate}
\subsection*{Organization of the Materials}
\par In \hyperref[sec:prelim]{Section~\ref{sec:prelim}}, we collect the necessary preliminaries, including the function spaces and basic analytic tools used throughout the paper. \hyperref[sec:semigroup-walks]{Section~\ref{sec:semigroup-walks}} introduces semigroup walks and develops the operator-theoretic framework. \hyperref[sec:ct-olric]{Section~\ref{sec:ct-olric}} is devoted to proving the existence and identification of Ollivier--Ricci curvature for admissible semigroup walks. \hyperref[sec:var]{Section~\ref{sec:var}} presents the variational characterization, and finally, \hyperref[sec:op-theo-ric]{Section~\ref{sec:op-theo-ric}} introduces the operator-theoretic Ricci curvature of operators. 
\addtocontents{toc}{\protect\setcounter{tocdepth}{-1}}
\section*{\small \bf  Acknowledgments}
\addtocontents{toc}{\protect\setcounter{tocdepth}{1}}
\begin{minipage}[t]{0.87\textwidth}
	\small
	\renewcommand{\labelitemi}{\raisebox{2pt}{\scalebox{1.2}{$\centerdot$}}}
		\begin{itemize}
			\item This work is based upon research funded by Iran National Science Foundation
			(INSF) under project No. 4030556.
	\end{itemize}
\end{minipage}
\normalsize
\vspace{10pt}
\section{Preliminaries}\label{sec:prelim}
\par Throughout this paper, $\left( \G, \dist \right)$ denotes a simple, $\dist$-locally finite graph. We further assume that $\left( \G, \dist \right)$ is a complete metric space.
\par It is important to note that for the special and frequently used class of weighted path metrics, the completeness is equivalent to the finiteness of all metric balls ($\dist$-local finiteness), by a Hopf--Rinow-type theorem~\cite{HKMW}. 
\subsection{Discrete calculus rudiments}\label{subsec:rudiments}
\par Conventionally, $x\sim y$ means that there is an edge between $x$ and $y$. 
\par For $x,y\in\G$, the directional derivative of $f$ from $x$ to $y$ is defined by
\begin{equation}\label{eq:derivative}
\nabla_{\vv{xy}} f := \dist(x,y)^{-1} \left( f(y) - f(x) \right).
\end{equation}
\par Given an edge weight function $\omega_{xy} := \omega(x,y)$, the general form of the discrete Dirichlet Laplacian in the finite regime is given by
\begin{equation*}
\Delta_{\omega, \mathdutchcal{g} , \m} f (x):= \m(x)^{-1}\left( \sum_{x \sim y} \left( f(y) - f(x)  \right) \omega_{xy} + \mathdutchcal{g}(x)f(x)\right),
\end{equation*}
where $\mathdutchcal{g}$ is a non-negative function~\cite{KLW}.
\par The special case of the combinatorial graph Laplacian $\Delta_{\sf c}$ is obtained by setting $\mathdutchcal{n} \equiv 1$, $\omega \equiv 1$, and $\mathdutchcal{g}\equiv0$. Another common operator is the normalized Laplacian, obtained by setting
$
\m(x) = \sum_{y\sim x} \omega_{xy}
$
and
$
\mathdutchcal{g}\equiv0
$.
\par \emph{The normalized Laplacian is particularly important because it serves as the infinitesimal generator of the standard weighted random walk which is the semigroup walk (in our terminology) associated with the heat semigroup; see, e.g.,~\cite{Barlow}.}
\subsection{Important (function) spaces}\label{subsubsec:imp-func-spaces}
\subsubsection{\small \textsf{\textbf{Standard spaces}}}
\par $\mathsf{C}_{\sf fs}(\G)$ denotes the space of finitely supported functions, $\Lip_{1}(\Omega)$, the space of $1$-Lipschitz functions on a subset $\Omega$, $\Lip_{1}(\Omega)$, on a subset $\Omega$, $\mathsf{C}_0(\G)$, the space of functions vanishing at infinity and $\mathsf{C}_{\sf b}$, the space of bounded functions. Note that $\fs$ coincides with the space of compactly supported functions, $\mathsf{C}_{\sf cs}$, since the distance, $\dist$, is assumed to be \emph{locally finite}.
\par The distance-weighted $p$-summable functions is defined by
\begin{equation*}
\ell_{\dist}^i(\G) := \ell^i(\G) \cap \LL^i(\G, \dist(z,\cdot)\texttt{\#}) = \LL^i(\G, \dist(z,\cdot)\texttt{\#}) \quad (\texttt{\#}: \text{counting measure}),
\end{equation*}
for some (and hence any) reference point $z \in \G$. The equality follows from the $\dist$-local finiteness. Similarly, we define $\ell^i_{\dist}$.
\par We denote the space of probability measures on $\G$ by $\mathscr{P}(\G)$, and let $\mathscr{P}_1(\G)$ denote the subspace of Borel probability measures with \emph{finite first moments}.
\subsubsection{\small \textsf{\textbf{Functions with reasonable Taylor remainders}}}\label{subsubsec:TR}
\par An important regularity that generalizes $\mathsf{C}^k$-smoothness is $\mathsf{TR}$-regularity. This space provides minimal conditions for uniqueness of normal forms for general random walks. 
\par The function space $\mathsf{TR}_0^{n+1,1}$ (TR stands for Taylor remainder), is defined as follows. Consider the class of functions
$
f: I \times \G^2 \to \R
$,
where $I$ is an open neighborhood of $0$, that are differentiable at $\varepsilon=0$ and admit an expansion of the form
\begin{equation*}
f(\varepsilon, z,w) = f(0,z,w) + \nicefrac{d}{d\varepsilon}\restr_{\varepsilon = 0} f(\cdot,z,w)\varepsilon + \cdots + \nicefrac{1}{n!} \; \nicefrac{d^n}{d\varepsilon^n} f(\cdot,z,w) \epsilon^n + R(\epsilon,z,w),
\end{equation*}
where, for all $z\in \G$,
\begin{equation}\label{eq:TR}
	f(0,z,\cdot),\quad \nicefrac{d}{d\varepsilon}\restr_{\varepsilon = 0}f(\cdot,z,\cdot), \quad \cdots, \quad \nicefrac{d^n}{d\varepsilon^n}\restr_{\varepsilon = 0} f(\cdot,z,\cdot)\quad  \text{and} \quad R(\epsilon,z,\cdot) \in \ell^1.
\end{equation}
Furthermore, we assume that the series
$
\varepsilon^{-1}\sum_{w\in \G} \left| R(\epsilon,z,w) \right|
$
converges locally uniformly in $\varepsilon$.
\par The function space $\mathsf{TR}_0^{n+1,1}$ then consists of the germs of such functions at $\varepsilon = 0$. With a slight abuse of notation, $f \in \mathsf{TR}_0^{n+1,1}$ means that the germ of $f$ belongs to $\mathsf{TR}_0^{n+1,1}$.
\par The distance-weighted spaces
$
{}^{\dist}\mathsf{TR}_0^{n+1,1} \subseteq \mathsf{TR}_0^{n+1,1}
$
are defined similarly, by requiring the functions in \hyperref[eq:TR]{\eqref{eq:TR}} to belong to $\ell^1_{\dist}$ and requiring the series
$
\varepsilon^{-1}\dist(z,\cdot)\sum_{w\in \G} \left|R(\epsilon,z,w)\right|
$
to converge locally uniformly in $\varepsilon$. \emph{Note that, on finite graphs, every $\mathsf{C}^{n+1}$ function $f$ (with domain of the form specified above) belongs to ${}^{\dist}\mathsf{TR}_0^{n+1,1}$.}
\subsubsection{\small \textsf{\textbf{General Banach spaces of functions}}}
\par On various occasions, we will need to work with general ``Banach spaces of functions'' not to be confused with ``Banach function spaces'' (which are Banach spaces satisfying the Bennett--Sharpley conditions).
\par Throughout this paper, $\B$ denotes a linear subspace of $\R^\G$, and $\|\cdot\|$ is a norm on $\B$ such that $\left(\B,\|\cdot\|\right)$ is a Banach space of functions. We further assume that $\B$ contains the finitely supported functions, $\fs$. The subspace of zero-mass signed measures in $\B$ will be denoted by $\B_0$. In general, $\B_0$ need not be closed in $\B$. 
\par A cone of functions $\mathsf{S}$ is a subset of $\R^\G$ satisfying 
$
0\in\mathsf{S}
$
and
$
\lambda\mathsf{S}\subseteq \mathsf{S}
$, 
for all $\lambda\geq0$.
\subsection{$\pmb{\mathsf{C}^{0}}$-semigroups of operators}\label{subsubsec:semigroups}
\par We will be dealing with $\mathsf{C}^{0}$-semigroups of linear operators $\T_t$, $t\geq 0$, on a Banach space of functions $\B$ on $\G$ with closed and densely defined infinitesimal generator $\mathcal A$.
\subsubsection{\small \textsf{\textbf{Taylor's theorem and $\pmb{k}$-regularity}}}
The Taylor-type formula with integral remainder
\begin{equation}\label{eq:Taylor}
\T_t f
=
\sum\limits_{i=0}^{n-1} \frac{t^i}{i!}\mathcal{A}^i f
+
\frac{1}{(n-1)!}\int_0^t (t-s)^{n-1}\T_s \; \mathcal{A}^n f \; ds,
\qquad
\forall f\in\mathsf{Dom}(\mathcal{A}^n),\quad t\geq 0,
\end{equation}
holds, where the integration is understood in the sense of Bochner. 
\par As a result, $f\in\mathsf{Dom}(\mathcal{A}^k)$ is equivalent to the orbit of $f$ being $\mathsf{C}^k$ and admitting a $\mathsf{C}^k$ extension beyond $t=0$. Consequently, when $f\in\mathsf{Dom}(\mathcal{A}^k)$, all the standard Banach-space $k$-term Taylor theorems, with their various forms of remainder, can be applied to $\T_t(f)$.
\par Evaluating both sides of \hyperref[eq:Taylor]{\eqref{eq:Taylor}} at a $w\in\G$, we can move the evaluation at $w$ inside the remainder integral to get a Riemann integral, since $\mathsf{C}^{0}$-semigroups are regulated; see~\cite[Proposition 3.7]{Cole}. Moreover, the Taylor series converges on a core consisting of analytic vectors.
\begin{definition}[$k$-regular semigroups]
	We say that a $\mathsf{C}^{0}$-semigroup $\T_\varepsilon$ is $k$-regular whenever $\mathsf{Dom}\left(\mathcal{A}^k\right)$ contains $\fs$. We call it $\infty$-regular if it is $k$-regular for every $k\in \mathbb{N}$.
\end{definition}
\subsubsection{\small \textsf{\textbf{Integral semigroups}}}\label{subsubsec:integral-smgrp}
\par If the $\mathsf{C}^{0}$-semigroup is given by integration against a time-dependent kernel $k(t,x,dy)$, we call it an \emph{integral semigroup}. Two important families of integral semigroups include Markov and Feller semigroups; see~\cite{BGL,Taira}.
\subsection{Random walks}
\par A random walk on $\G$ consists of a collection of Borel probability measures $\upmu_z$, $z\in\G$, with \emph{finite first moments}; that is, $\upmu_z\in\mathscr{P}_1(\G)$. On metric measure spaces, one also imposes the measurable dependence on the base point $z$; a condition that is automatically satisfied in the discrete setting. The measure $\upmu_z$ determines, once at $z$, where and with what probability a drunkard walking on the graph moves next.
\par In order to capture the geometry of the space at different scales, one must consider a one-parameter family of random walks. \emph{In this paper, we work with random walks 
	$
	\upmu_z^\varepsilon
	$,
	for $0\le \varepsilon\le 1$ that satisfy the natural and geometrically intuitive condition $\upmu_z^0=\delta_z$ (standing assumption throughout).}
\begin{definition}[Lazy random walks]\label{defn:lazy-walk}
	A random walk $\upmu_z^\varepsilon$ is called \emph{lazy} if $
		\operatorname{supp}(\upmu_z^\varepsilon) \subset \mathrm{B}_1(z)
	$, for all $z\in\G$ and $\eps\ge 0$; here, $\mathrm{B}_1(z)$ is the combinatorial unit ball ($z$ and all vertices connected to $z$). 
\end{definition}
\subsubsection{\small \textsf{\textbf{Markovian walks}}}
\par Markovian walks are particularly important from both geometric and analytic points of view. Analytically, every Markov semigroup $\T_t$ on the Banach space $\mathsf{C}_{\sf b}(\G)$ is the transition semigroup of a (normal càdlàg) Markov process on $\G$, with transition probabilities given by
$
	p(t,z,dy)=\mathrm{T}_t(\delta_z)(y)
$;
see, e.g., \cite{Taira}. Geometrically, the coarse Ricci curvature of certain Markovian walks, such as heat-kernel random walks, reflects geometric properties of the underlying graph; see, e.g., \cite{Ol,MW}, among many other works in this direction.
\subsection{Coarse Ollivier--Ricci curvature}\label{subsubsec:ctoRic}
\par The coarse Ollivier--Ricci curvature associated with a one-parameter family of random walks between $x$ and $y$ ($x\neq y$) is defined, following Ollivier~\cite{Ol}, by
\begin{equation*}
\Ric_\varepsilon(x,y)
:=
1-\frac{\Was_1\left(\upmu_x^\varepsilon,\upmu_y^\varepsilon\right)}
{\dist(x,y)},
\qquad 0\leq\varepsilon\leq 1,
\end{equation*}
where $\Was_1$ denotes the $L^1$-Wasserstein distance. 
\par In the setting of Riemannian manifolds with $\Ric\geq K$, and using uniform measures $\upmu_z^\varepsilon$ on $\B_\varepsilon(z)$ as $\varepsilon$-step random walks, the coarse Ricci curvature exhibits the following precise asymptotic behavior in terms of distance and Ricci curvature:
\begin{equation}\label{eq:cto-riem}
	\Ric_{\varepsilon}(x,y)
	=
	\nicefrac{\varepsilon^2}{2(N+2)}
	\,\Ric\left(
	\frac{\vv{xy}}{\|\vv{xy}\|},
	\frac{\vv{xy}}{\|\vv{xy}\|}
	\right)
	+\mathcal{O}\left(\varepsilon^3 + \eps^2 \dist(x,y)\right);
\end{equation}
as
$
(\eps, \dist(x,y)) \downarrow (0,0)
$. Here, $\vv{xy}$ denotes the vector in the tangent space satisfying $\exp_x(\vv{xy})=y$~\cite{Ol}.
\subsection{Ollivier--Ricci curvature}\label{subsubsec:ct-olric}
\par The Ollivier--Ricci curvature $\Ric(x,y)$ is defined, following Ollivier, as the derivative
\begin{equation*}
\Ric(x,y)
:=
-\dist(x,y)^{-1}
\;\nicefrac{d^+}{d\varepsilon}\restr_{\varepsilon=0}
\;\Was_1\left(\upmu_x^\varepsilon,\upmu_y^\varepsilon\right),
\qquad \text{for }x\neq y,
\end{equation*}
\emph{whenever the one-sided derivative exists}~\cite{Ol}. The original definition is intended for a random walk generated by a Markov kernel. However, the same definition can be used for more general one-parameter families of random walks. We note that, for the definition of Ollivier--Ricci curvature, we only need to know the germ of the random walk at $\varepsilon=0$.
\par To gain more intuition on the Ollivier--Ricci curvature, it is important to note that, in the Riemannian setting, by \hyperref[eq:cto-riem]{\eqref{eq:cto-riem}}, for uniform-measure random walks (or semigroup random walk given by the heat flow), $\Ric$ would be identically zero. Yet, upon nonlinearly rescaling the time parameter by taking its square root, for the random walk
$
\upeta_z^\eps :=  \upmu_z^{\sqrt{2(n+2)\eps}}
$,
it holds
\begin{equation*}
\Ric(x,y) := \Ric\left(
\frac{\vv{xy}}{\|\vv{xy}\|},
\frac{\vv{xy}}{\|\vv{xy}\|}
\right);
\end{equation*}
thus, one recovers the actual Ricci curvature. This means, infinitesimally, certain natural random walks on a manifold and on Euclidean space give rise to the same Ollivier--Ricci curvature reflecting the fact that infinitesimally (by the Taylor expansion of the Riemannian exponential map), geodesics on Riemannian manifolds are Euclidean up to the second order. Hence, one must rescale time to capture the second-order term governed by curvature. 
\subsubsection{\small \textsf{\textbf{Upper/Lower Ollivier--Ricci curvature}}}
\par We define the upper Ollivier--Ricci curvature, with values in the extended real numbers, by the upper Dini derivative at $\eps= 0$,
\begin{equation*}
\Ric^+
:=
\varlimsup_{\epsilon\downarrow0} \;
\nicefrac{1}{\eps}
\left(
1-\frac{\Was_1\left(\upmu_x^{\varepsilon},\upmu_y^{\varepsilon}\right)}
{\dist(x,y)}
\right),
\end{equation*}
and define the lower Ollivier--Ricci curvature, $\Ric^-$, similarly by the lower Dini derivative at $\eps = 0$. In general, we have
$
\Ric^-\neq\Ric^+
$;
see~\cite{Veys} for some examples.
\section{Operator theory preparation and Semigroup walks}\label{sec:semigroup-walks}
\subsection{Operator-theoretic tools and concepts}\label{sec:opt-rw}
\subsubsection{\small \textsf{\textbf{Needed operator-theoretic properties}}}\label{subsubsec:sym-loc}
\par Throughout this section, 
$
\mathcal{A}: \R^\G\supseteq \mathsf{Dom}(\mathcal{A}) \to \R^\G
$
is a linear map. We assume the domain and range of $\mathcal{A}$ are equipped with norms. 
\begin{definition}[(rough) quasi-symmetry]\label{defn:quasi-sym}
	\hfill
	\begin{itemize}
		\item Suppose
		$
		1\leq\Uptheta<\infty
		$. 
		$\mathcal{A}$ is said to be \emph{roughly $\Uptheta$-quasi-symmetric} if
		\begin{equation}\label{eq:quas-sym}
		\frac{1}{\Theta}\left|\left<f,\mathcal{A}g\right>\right|\le \left|\left<\mathcal{A}f,g\right>\right|
			\leq
			\Uptheta\left|\left<f,\mathcal{A}g\right>\right|,
			\qquad
			\forall f,g\in  \mathsf{Dom}(\mathcal{A}) \cap \fs,
		\end{equation}
		holds with respect to the inner product in $\ell^2(\G)$.
		\smallskip
		\item When
	$
			\left<\mathcal{A}f,g\right>
			=
			\left<f,\mathcal{A}g\right>
	$, for all $f,g\in\mathsf{Dom}(\mathcal{A})\cap\fs$,
		we call $\mathcal{A}$ \emph{roughly symmetric}.
		\smallskip
		\item A \emph{symmetric operator}
		$
		\mathcal{A}:\mathsf{Dom}(\mathcal{A})\subseteq \ell^2(\G)\to\ell^2(\G)
		$
		is an operator with a dense domain satisfying the above symmetry relation for all $f,g$ in its domain. If, instead, the inequality~\hyperref[eq:quas-sym]{\eqref{eq:quas-sym}} holds for all $f,g$ in the domain of $\mathcal{A}$, we call $\mathcal{A}$ a \emph{$\Uptheta$-quasi-symmetric operator}.
	\end{itemize}
\end{definition}
\par A simple observation is that the rough quasi-symmetry implies that for any fixed $g \in \mathsf{Dom}(\mathcal{A}) \cap \fs$, we have
\begin{equation*}
\sup_{f \in \fs \cap \Dom(\mathcal A)} \left<\mathcal Af, g \right> \le \Theta \sup_{f \in \fs \cap \Dom(\mathcal A)} \left<f, \mathcal A g \right>.
\end{equation*}
\begin{definition}[Semi-local operator]\label{defn:semi-loc}
	Let
$
R:\G\to\R^+
$
be a positive radius function. 
\begin{itemize}
	\item $\mathcal{A}$ is said to be $R$-semi-local whenever, for each $x_0$ and $f\in\mathsf{Dom}(\mathcal{A})$,
	\begin{equation*}
		f\restr_{\B_{R(x_0)}(x_0)}=0
		\quad \Longrightarrow\quad 
		\mathcal{A}(f)(x_0)=0.
	\end{equation*}
	\emph{$\mathcal A$ is said to be semi-local whenever it is semi-local with respect to some positive function $R$.}
	\item When $R\equiv R_0$ is constant, we say $\mathcal A$ is uniformly semi-local at scale $R_0$.
	\smallskip
	\item If $\mathcal A$ is $R$-semi-local and $\mathcal A(\fs) \subseteq \fs$, then we call $\mathcal A$, a strongly $R$-semi-local operator. 
	\smallskip
	\item $\mathcal A$ is said to be semi-local at a controlled scale whenever there exists a non-decreasing function
	$
	s: \R^{\ge 0} \to \R^{\ge 0}
	$,
	such that
	$
	\left|R(y) - R(x)\right| \le s(\dist(x,y))
	$ for all  $x,y\in \G$. 
\par Clearly, a uniformly semi-local operator is both strongly semi-local and semi-local at a controlled scale.
\end{itemize}
\end{definition}
\par A simple yet important observation is that if the $\mathsf{C}^{0}$-semigroup $\T_\epsilon$ is generated by a strongly semi-local operator $\mathcal{L}$, then the semigroup is $k$-regular for some $k\ge 1$ if and only if it is $\infty$-regular.
\par A semi-local map $F$ (not necessarily linear) on a subset $D\subseteq \Dom(F)$ is defined in the same fashion.
\begin{lemma}[Lipschitz tapering]\label{lem:Lip-taper}
	Suppose $\mathcal{L}$ is a semi-local operator at a controlled scale. Let $x\in\G$ and $M>0$ be fixed. Then, for sufficiently large $\Uplambda>0$ and 
	\begin{equation*}
	\zeta > M + 2\Uplambda + s(\Uplambda),
	\end{equation*}
the following holds. For any $f\in\Lip_1(\G)$ with $f(x)\le M$, there exists
$
	\widehat{f}\in\fs\cap\Lip_1(\G)
$
	such that
	\begin{equation*}
	\widehat{f}\restr_{B_{2\Uplambda + s(\Uplambda)}(x)} = f\restr_{B_{2\Uplambda + s(\Uplambda)}(x)}, \qquad \left(\mathcal{L} \widehat{f}\,\right) \restr_{B_{\Uplambda}(x)} = \left(\mathcal{L}f\right) \restr_{B_{\Uplambda}(x)},
	\end{equation*}
	and
	$
	\widehat{f}\equiv0
$ on 
$
	\G\smallsetminus B_{2\zeta}(x)
$.
\end{lemma}
\begin{proof}[\footnotesize\textbf{Proof}]
	Pick $\Uplambda > R(x)$ and pick $\zeta$ so that
\begin{equation*}
		\zeta > M + 2\Uplambda + s(\Uplambda).
\end{equation*} 
\par Take the cut-off function
	\begin{equation*}
	\upchi_{_{\zeta}}(z)
	:=
	\Big(\zeta\wedge\big(2\zeta-\dist(x,z)\big)\Big) \vee 0.
	\end{equation*}
It is easy to check that $\upchi_{_{\zeta}}$ is $1$-Lipschitz and equals $\zeta$ on $B_{\zeta}(x)$. 
\par We deduce that $\upchi_{_{\zeta}}\wedge f$ coincides with $f$ on $B_{2\Uplambda + s(\Uplambda)}(x)$. Indeed, for $z\in B_{2\Uplambda + s(\Uplambda)}(x)$, one easily obtains
	\begin{equation*}
	f(z)
	\leq
	f(x)+\dist(x,z)
	=
	M + 2\Uplambda + s(\Uplambda)
	< \zeta.
	\end{equation*}
We set
	$
	\widehat{f}
	:=
	\left(-\upchi_{_{\zeta}} \vee f\right)\wedge \upchi_{_{\zeta}}
	$,
	which is clearly $1$-Lipschitz and satisfies
	\begin{equation}\label{eq:f-coincide}
		\widehat{f}\restr_{B_{2\Uplambda + s(\Uplambda)}(x)}
		=
		f\restr_{B_{2\Uplambda+ s(\Uplambda)}(x)}.
	\end{equation}
\par The controlled semi-locality of $\mathcal{L}$ implies that
\begin{equation*}
R(y) \le R(x) + s(\Uplambda)\le \Uplambda + s(\Uplambda), \quad \forall y \in B_{\Uplambda}(x);
\end{equation*}
thus from \hyperref[eq:f-coincide]{\eqref{eq:f-coincide}}, we deduce that for any such $y$, $\widehat{f}$ and $f$ coincide on 
$B_{R(y)}(y)$ and therefore, 	
	$
		\mathcal{L}\widehat{f}
		=
		\mathcal{L}f
	$
	on 
	$
	B_{\Uplambda}(x)
	$.
\end{proof}
\begin{definition}[Divergence-type operators]\label{defn:div-type}
	$\mathcal{A}$ is said to be \emph{weakly of divergence type} with respect to a measure $\m$ if
		\begin{equation*}
		\int_\G \mathcal{A}f\,d\m=0,
		\qquad
		\forall f\in\mathcal{C}_{\sf fs}(\G)\cap\mathsf{Dom}(\mathcal{A}).
		\end{equation*}
In particular, if $\m$ is an invariant measure for a $\mathsf{C}^0$-semigroup $\mathrm{T}_t$ generated by $\mathcal{A}$, then $\mathcal{A}$ is a divergence-type operator with respect to $\m$.	
\par If $\mathcal{A}$ is weakly of divergence type with respect to the counting measure $\m\equiv1$, we simply call $\mathcal{A}$ \emph{weakly of divergence type}.
\end{definition}
\begin{definition}[Comparable to multiplication]\label{defn:rough-comp-mult}
	We say that an operator $\mathcal{A}$ is \emph{comparable to multiplication} if there exists a non-negative function $q(x)$ such that
	\begin{equation*}
	\left|\mathcal{A}f(x)\right|
	\leq
	q(x)\|f\|_{\infty},
	\qquad
	\forall f\in\Dom(\mathcal{A}) \cap \cb.
	\end{equation*}
\end{definition}
\begin{definition}[Partial order]\label{def:order}
	Suppose $\mathcal{A}$ and $\mathcal{B}$ are two operators. We write
	\(
	\mathcal{A}\geq\mathcal{B}
	\)
	whenever
	\(
	\mathcal{A}(f)\geq\mathcal{B}(f)
	\)
	for all non-negative functions $f$ satisfying
	\(
	f\in\mathsf{Dom}(\mathcal{A})\cap\mathsf{Dom}(\mathcal{B}).
	\)
\end{definition}
\begin{remark}
	\textsl{The more or less standard terminology is that when $\mathcal{A} \ge 0$, one calls it a positive operator however, one can distinguish between non-negativity preserving operators, $\mathcal A \ge 0$, and positivity improving ones that are defined via
	\[
	\mathcal{A} f > 0, \quad \forall 0 \le f\in{\Dom(\mathcal{A})} \; \text{and}\; f \not\equiv 0;
	\]
but, in these notes, we will not need the latter more refined concepts.}
\end{remark}
\subsubsection{\small \textsf{\textbf{Conical-supremum extension}}}\label{subsubsec:conical-sup}
\begin{definition}[Conical supremum extension of non-negative operators]
	Suppose $\mathcal{A}$ is a positive operator and $\mathsf{S}$ is a cone of functions (see~\hyperref[subsubsec:imp-func-spaces]{Section~\ref{subsubsec:imp-func-spaces}}). The conical-sup extension $\widetilde{^\mathsf{S}\mathcal{A}}$ of $\mathcal{A}$ with respect to $\mathsf{S}$ is defined by
	\begin{equation*}
	\widetilde{^\mathsf{S}\mathcal{A}}(f)(x)
	:=
	\sup_{\substack{g\in \mathsf{Dom}(\mathcal{A})\cap \mathsf{S}\\ g\leq f}}
	\mathcal{A}g(x),
	\end{equation*}
	with values in the extended real numbers. \emph{The domain of $\widetilde{^\mathsf{S}\mathcal{A}}$ consists of those functions for which the right-hand side is finite for all $x\in\G$.} The map $\widetilde{^\mathsf{S}\mathcal{A}}$ is an extension of $\mathcal{A}\restr_{\mathsf{Dom}(\mathcal A) \cap \mathsf{S}}$.
	\par \emph{Note that the domains of these extensions are conical but not necessarily linear, and the conical-sup extensions are not necessarily linear.}
\end{definition}
\par It is easily verified that
$
\widetilde{^\mathsf{S}\mathcal{A}}(\lambda f)
=
\lambda\,\widetilde{^\mathsf{S}\mathcal{A}}(f)
$, for all $\lambda \ge 0$; also the conical-sup extension is positive and indeed monotone in the sense that
$
	f_1\leq f_2
$
 implies 
$
	\widetilde{^\mathsf{S}\mathcal{A}}\left(f_1\right)
	\leq
	\widetilde{^\mathsf{S}\mathcal{A}}\left(f_2\right)
$.

\par If $\mathsf{S}$ is a linear space, then
	\begin{equation*}
	\widetilde{^\mathsf{S}\mathcal{A}}\left(f_1+f_2\right)
	\geq
	\widetilde{^\mathsf{S}\mathcal{A}}\left(f_1\right)
	+
	\widetilde{^\mathsf{S}\mathcal{A}}\left(f_2\right);
	\end{equation*}
	in particular, for $f_1,f_2\geq0$,
	\begin{equation*}
	f_1+f_2\in
	\mathsf{Dom}\left(\widetilde{^\mathsf{S}\mathcal{A}}\right)
	\quad \Longrightarrow \quad
	f_1,f_2\in
	\mathsf{Dom}\left(\widetilde{^\mathsf{S}\mathcal{A}}\right).
	\end{equation*}
	When
	\begin{equation*}
	\mathsf{S}\cap\mathsf{Dom}(\mathcal A)\subseteq\fs,
	\qquad\text{or}\qquad
	\mathsf{S}\cap\mathsf{Dom}(\mathcal A)\subseteq\mathsf{C}_{\sf b},
	\end{equation*}
	we have
	\begin{equation*}
	f_1+f_2\in
	\mathsf{Dom}\left(\widetilde{^\mathsf{S}\mathcal{A}}\right) \quad 
	\iff \quad 
	f_1,f_2\in
	\mathsf{Dom}\left(\widetilde{^\mathsf{S}\mathcal{A}}\right),
	\end{equation*}
	and
	\begin{equation*}
	\widetilde{^\mathsf{S}\mathcal{A}}(f_1+f_2)
	=
	\widetilde{^\mathsf{S}\mathcal{A}}(f_1)
	+
	\widetilde{^\mathsf{S}\mathcal{A}}(f_2),
	\qquad
	\text{for }f_1,f_2\geq0.
	\end{equation*}
\begin{definition}[Natural cone]\label{defn:natural-cone}
	Let $\mathcal A$ be a given operator. In view of the useful properties it provides, we say that a cone $\mathsf{S}$ is \emph{natural} with respect to $\mathsf{Dom}(\mathcal A)$ whenever
	\begin{equation*}
	\cb^{\geq0}
	\subseteq
	\mathsf{S}\cap\mathsf{Dom}(\mathcal A)
	\subseteq
	\cb.
	\end{equation*}
	\emph{When the domain of the operator clear from the context, we simply call $\mathsf{S}$ a natural cone.}
\end{definition}
\par As we will see later, natural cones behave much like $\mathsf{S}=\cb$.
\par A straightforward observation is that if $\T$ is a positive integral operator with kernel $k(x,y)$ and $\mathsf{S}$ is a natural cone with respect to $\mathsf{Dom}(\T)$, then for any $f\geq0$,
$
\widetilde{^\mathsf{S}\T}(f)(x)<\infty
$
implies 
\begin{equation}\label{eq:int-ext}
\widetilde{^\mathsf{S}\T}(f)(x)
=
\int k(x,dy)f(y).
\end{equation}
In other words, the integral representation extends to $f$.
\par This follows from the fact that, by definition, $f\geq0$ can be approximated pointwise by an increasing sequence in
$
\mathsf{S}\cap\mathsf{Dom}(\T)
$,
and hence applying the monotone convergence theorem yields \hyperref[eq:int-ext]{\eqref{eq:int-ext}}. 
\subsubsection{\small \textsf{\textbf{Integrated comparison}}}\label{subsubsec:consup-comp}
\begin{definition}[Positive semigroups with bounded below generators]\label{defn:bounded-below}
	We say the generator $\mathcal{L}$ of a positive $\mathsf{C}^0$-semigroup $\T_t:\B\to\B$ is bounded below if there exists a function $q \in \R^\G$ that generates a multiplication semigroup $e^{tq}$ on $\B$ (need not be strongly continuous) such that $\Lo \ge q$ (see \hyperref[def:order]{Definition~\ref{def:order}}).
\end{definition}
\par If a positive smeigorup $\T_t$ has a bounded below generator, then
\begin{equation*}
\T_{t_1}
\geq
e^{t_0q}\,\T_{t_1-t_0},
\qquad
0\leq t_0\leq t_1.
\end{equation*}
This follows from the positivity of $\T_t$ and
\begin{equation*}
\nicefrac{d}{dt}\,\mathrm{P}_t
=
e^{-tq}\left(\mathcal{L}-q\right)\T_t,
\qquad
\text{where }
\mathrm{P}_t:=e^{-tq}\,\T_t.
\end{equation*}
\par If furtheremore,
\begin{equation*}
\widetilde{^\mathsf{S}\T}_{t_0}(f)(x_0)<\infty,
\qquad
\text{for some }t_0>0\text{ and }x_0\in\G.
\end{equation*}
 then, by the argument above, for any
\begin{equation*}
g\in\B\cap\mathsf{S},
\qquad
g\leq f,
\end{equation*}
and for $t\leq t_0$, we have
\begin{equation*}
\widetilde{^\mathsf{S}\T}_{t_0}f(x_0)
\geq
\T_{t_0}g(x_0)
\geq
e^{(t_0-t)q(x_0)}
\T_tg(x_0).
\end{equation*}
\par Upon taking the supremum over $g$, we obtain
\begin{equation*}
\widetilde{^\mathsf{S}\T}_{t}f(x_0)
\leq
e^{-(t_0-t)q(x_0)}
\widetilde{^\mathsf{S}\T}_{t_0}f(x_0).
\end{equation*}
In particular, one must have
\begin{equation*}
\sup_{t\in[0,t_0]}
\widetilde{^\mathsf{S}\T}_t f(x_0)
<\infty.
\end{equation*}
\subsubsection{\small \textsf{\textbf{Nonlinear mappings with uniform asymptotics along a subset}}}\label{subsubsec:unif-asymp}
\par Let $\B_1$ be a Banach space and $\B_2$ a Banach space of functions on $X$. Suppose $U\subseteq \B_1$ is open and $\mathsf A\subseteq \B_2$. Let 
\begin{equation*}
F:U\to \mathrm{Aut}(\B_2,\B_2)
\end{equation*}
be a continuous nonlinear mapping. Here, $\mathrm{Aut}(\B_2,\B_2)$ is the Banach space of bounded (continuous) automorphisms equipped with the operator norm. 
\begin{definition}
	We say that $F$ has \emph{uniform asymptotics of order $n+1$ at $a\in U$} along $\mathsf A$, if there exist operators
	\begin{equation*}
	L^a_i:\B_1^i\to \mathrm{Aut}(\B_2,\B_2),
	\end{equation*}
	such that
	\begin{equation*}
	\mathcal{R}_{a,h}
	:=
	F(a+h)
	-
	\Big(
	F(a)+L^a_1h+L^a_2(h^2)+\cdots+L^a_n(h^n)
	\Big),
	\end{equation*}
	satisfies
	\begin{equation*}
	\mathcal{R}_{a,h}f(z)
	=
	\smallO(\|h\|^n),
	\qquad
	\forall f\in \mathsf A,
	\quad
	\forall z\in X,
	\end{equation*}
	as $h\to0$, and furthermore \emph{the error is uniform in $f$} (notice that the error is not necessarily uniform in $z$).
\end{definition}
\par For example, a sufficient condition for $F$ to have uniform asymptotics of order $n+1$ at $a$ is that $F$ be Fr\'echet $\mathsf{C}^{n-1}$ and $n$-times Fr\'echet differentiable at $a$; see~\cite[Theorem 5.6.3]{Cartan}. Notice that, when $n>2$, there are cases in which $F$ has such uniform asymptotics without being $n$-times differentiable; see~\cite[Section 5.2]{Cole}.
\begin{remark}\label{rem:uniform-diff}
	\par One immediate observation is that, when $\B_1=\R$, $X$ is a metric space, and
	$
	\Lip_1(X)\subset\B_2
	$,
	the following holds. Suppose $F(t)$ is differentiable at $0$, $F'(0)$ is continuous with respect to the point-open topology, and $F(t)$ has uniform second-order asymptotics at $t = 0$ along $\Lip_1(X)$. If
	\begin{equation*}
	\Lip_1(X)\ni f_\epsilon
	\xrightarrow[\text{pointwise}]{} f_0,
	\end{equation*}
	then
	\begin{equation*}
	\lim_{\epsilon\to0}
	\frac{1}{\epsilon}
	\big(F(\epsilon)-F(0)\big)f_\epsilon(z)
	=
	F'(0)f_0(z).
	\end{equation*}
Recall that continuity with respect to the point-open topology implies continuity with respect to pointwise convergence of functions in the domain.
\end{remark}
\subsection{Semigroup random walks}\label{subsec:rw}
\hfill\\
\dunderline{1pt}{\raisebox{1pt}{\small\textbf{\textsf{Convention:}}}} \textsl{In the sequel, we assume that the domains of the operators encountered contain $\fs$, unless otherwise stated or implied.}
\par Let us introduce the following class of random walks represented by $\mathsf{C}^0$-semigroups of operators.
\begin{definition}[Semigroup random walk]
	\hfill
	\begin{itemize}
	\item \textbf{\small Forward semigroup walk:} Suppose $\mathrm{T}_t:\B\to\B$ is a $\mathsf{C}^0$-semigroup of operators such that, for each $z$, $\mathrm{T}_t(\delta_z)$ is a non-negative probability measure with finite first moment (i.e., finite-support probability densities are mapped to probability densities with finite first moments). Then we call the random walk given by
	$
	\upmu_z^\epsilon(y)=\mathrm{T}_\epsilon(\delta_z)(y)
	$,
	a \emph{forward semigroup walk}. 
\item \textbf{\small Backward semigroup walk:}
	There is a dual picture as well. If the mass preservation
	\begin{align}\label{eq:mass-preserv}
	\sum_{y\in\G}\mathrm{T}_\epsilon(\delta_y)\equiv 1,\qquad \forall\,\epsilon,
	\end{align}
	holds, then
$
	\upmu_z^\epsilon(y)=\mathrm{T}_\epsilon(\delta_y)(z)
$
	also defines a random walk, provided that it has finite first moments. We call the latter a \emph{backward semigroup walk}.
	\end{itemize}
\end{definition}
\par When $\mathrm{T}_\epsilon$ is roughly symmetric, the forward and backward semigroup walks coincide. 
\par An important class within the category of semigroup walks consists of integral semigroup walks, i.e., those given by transition kernels; see \hyperref[subsubsec:integral-smgrp]{Section~\ref{subsubsec:integral-smgrp}}. This class of random walks contains Markovian and Feller walks. 
\subsubsection{\small \textsf{\textbf{Formal stochastic completeness}}}\label{subsubsec:stoch-complete}
\begin{definition}[Formal stochastic completeness]\label{defn:formal-stoch-compl}
	Suppose $\B$ is a Banach space of functions containing constant functions. A semigroup $\T_t: \B \to \B$ is called formally stochastically complete if 
$
	\T_t(1) = 1
$
\end{definition}
\begin{remark}
	\textsl{By the standard compact-space theory, when $\G$ is a finite graph, a random walk defined on $\mathsf{C}_0(\G)=\R^\G$ is Feller and $\T_t(1)=1$ if and only if it is Markovian. See, for example, \cite[Proposition 2.3.2]{ACLR}.} The condition $\T_t(1)=1$ is called stochastic completeness.
\end{remark}
\begin{remark}
	The mass preservation property \hyperref[eq:mass-preserv]{\eqref{eq:mass-preserv}} is closely related to the stochastic completeness of the defining semigroup; i.e.,
	$
	\T_\eps(1) = 1
	$.
	\par For example, in the case of heat flows, stochastic completeness follows from the property that Ricci curvature is bounded below; see \cite[Theorem 4.11]{MW}. In our setting, if $\T_\eps$ is positive and continuous w.r.t. the point-open topology, then the mass preservation and stochastic completeness are equivalent. The same also holds if $\T_\eps$ is a positive integral semigroup.
\end{remark}
\par \emph{For a treatment of Markovian semigroups as gradient flows of Dirichlet forms, we refer the reader to \cite{KLW}. In this paper, our approach is more general and works directly with the semigroup rather than viewing the semigroup as a gradient flow.}
\subsubsection{\small \textsf{\textbf{The infinitesimal generator}}}\label{subsubsec:inf-gen}
\par Suppose $\mathcal{U}$ is the generator of the semigroup $\T_t$. If $\T_\eps(\delta_z)$ defines a forward semigroup walk and if $\B_0$ is a closed subspace of $\B$ (hence Banach), then $\mathcal{U}$ is weakly of divergence type with respect to the counting measure, i.e.,
\begin{equation*}
\sum_{z \in \G}\mathcal{U}f(z)=0,\qquad
\forall\,f\in\fs\cap\mathsf{Dom}(\mathcal{U});
\end{equation*}
see \hyperref[defn:div-type]{Definition~\ref{defn:div-type}}. Consequently, if, furthermore, $\T_\varepsilon(\delta_z)$ is a $\mathsf{C}^{k+1}$ walk, then
$
\mathcal{U}^i
$, $1\le i \le k$,
are also weakly of divergence type with respect to the counting measure. In particular, for an admissible forward semigroup walk, $\Lo$ is weakly of divergence type.
\par For a backward semigroup walk, $\upmu_z^\eps(w) = \T_\eps(\delta_w)(z)$, by the formal stochastic completeness, we deduce
$
\Lo 1 = 0
$;
i.e., $\Lo$ is translation invariant. 
\section{Ollivier--Ricci curvature of semigroup walks}\label{sec:ct-olric}
\subsection{Properties of $\mathbf{Ric}_{\scalebox{0.7}{$\pmb{\varepsilon}$}}$}\label{sec:dt-olric}
Recall that, for a one-parameter family of random walks $\upmu^\varepsilon_z$, the corresponding one-parameter family of coarse Ollivier--Ricci curvatures is given by
\begin{equation*}
\Ric_\varepsilon(x,y)
:=
1-\frac{\Was_1\left(\upmu^\varepsilon_x,\upmu^\varepsilon_y\right)}
{\dist(x,y)}.
\end{equation*}
\par The following properties follow from the, by now standard, theory of optimal transportation; see~\cite{Vil}.
\begin{itemize}
	\item (Restriction of Kantorovich-Rubinstein formulation) For any two probability measures $\upmu$ and $\upnu$ with finite first moments, one has
	\begin{align*}
		\Was_1\left(\upmu,\upnu\right)
		&=
		\sup_{f\in\Lip(1)\cap\mathcal{C}_{\sf fs}(\G)}
		\sum_{z\in\G}
		f(z)\left(\upmu(z)-\upnu(z)\right)\\
		&= \sup_{f\in\Lip(1)\cap\mathcal{C}_{\sf b}(\G)}
		\sum_{z\in\G}
		f(z)\left(\upmu(z)-\upnu(z)\right).
	\end{align*}
	\item  (Localization of Kantorovich-Rubinstein formulation)
		\begin{equation*}
		\Was_1\left(\upmu_x^\varepsilon,\upmu_y^\varepsilon\right)
		=
		\sup_{f\in\Lip_1\left(\Upomega^\varepsilon_{xy}\right)}
		\sum_{z\in\Upomega^\varepsilon_{xy}}
		f(z)
		\left(
		\upmu_x^\varepsilon(z)-\upmu_y^\varepsilon(z)
		\right),
	\end{equation*}
	where 
	\begin{equation*}
	\Upomega^\varepsilon_{xy} = \supp(\upmu_x^\eps) \cup \supp(\upmu_y^\eps).
\end{equation*}
	\item (Maximizer in the Kantorovich-Rubinstein formulation) There always exist a maximizer $f_{\max} \in \Lip_1(\G)$ for the Kantorovich-Rubinstein.
	\item (Localization of $\Ric_\varepsilon$) $\Ric_\varepsilon(x,y)$ coincides with the coarse Ollivier--Ricci curvature on the subgraph $\Upomega^\varepsilon_{xy}$.
	\item (Concavity of $\Ric_\varepsilon$)  For two random walks $\upmu_z$ and $\upnu_z$, we have
		\begin{equation*}
			\Ric^{(1-t)\upmu+t\upnu}(x,y)
			\geq
			(1-t)\Ric^{\upmu}(x,y)
			+
			t\Ric^{\upnu}(x,y).
		\end{equation*}
			In particular, if $\upeta_z^\varepsilon$ is a time-affine walk, then
		$\Ric^{\upeta}_\varepsilon(x,y)$ is a concave function of $\varepsilon$.
\end{itemize}
\subsection{Well-definedness of ${\mathbf{Ric}}$ for time-affine random walks}
\begin{lemma}\label{lem:wd-affine}
	$\Ric(x,y)$ is well-defined for all time-affine random walks of the form
	\begin{equation*}
	\upmu_z^\varepsilon(y) = \updelta_z(y) + \varepsilon \upmu_z(y),
	\end{equation*}
	where $\upmu_z$ is a zero-mass signed measure with finite first moment.
\end{lemma}
\begin{proof}[\footnotesize \textbf{Proof}]
	From the hypotheses, we obtain
	\begin{align*}
		\Was_1\left( \updelta_z, \upmu_z^\varepsilon\right)
		&\leq \int \dist(z,w)\,d\upmu_z^\varepsilon(w)\\
		&= \int_{\G\smallsetminus{z}} \dist(z,y)\,d(\updelta_z+\varepsilon\upmu_z)\\
		&= \varepsilon\int_{\G} \dist(z,y)\,d\upmu_z\\
		&= O(\varepsilon).
	\end{align*}
	Since $\Was_1$ metrizes $\mathscr{P}_1(\G)$, we deduce
	\begin{align*}
		\left|\nicefrac{1}{\varepsilon}\,\Ric_\varepsilon(x,y)\right|
		&=
		\left|\nicefrac{1}{\varepsilon}
		\left(
		1-\frac{\Was_1\left( \upmu_x^\varepsilon,\upmu_y^\varepsilon\right)}
		{\dist(x,y)}
		\right)\right|\\
		&\le
		\nicefrac{1}{\varepsilon}
		\left(
		\frac{\Was_1\left(\updelta_x,\upmu_x^\varepsilon\right)}
		{\dist(x,y)}
		+\frac{\Was_1\left(\updelta_y,\upmu_y^\varepsilon\right)}
		{\dist(x,y)}
		\right).
	\end{align*}
Thus, we deduce $\nicefrac{1}{\varepsilon}\,\Ric_\varepsilon(x,y)$ is non-negative function that admits a finite lower bound as $\varepsilon\downarrow0$.	
\par Since $\Ric_\eps$ is concave in $\eps$, and $\Ric_0 \equiv 0$, we deduce
\[
\nicefrac{1}{\eps} \Ric_\eps \le \nicefrac{1}{\tau} \Ric_{\tau}, \quad \text{for all} \quad \tau \le \eps.
\]	
Thus 
$
\nicefrac{1}{\eps} \Ric_\eps
$
is a bounded non-increasing function of $\eps$ and we deduce
$
	\lim_{\varepsilon\downarrow0}
	\nicefrac{1}{\eps}\,\Ric_\varepsilon
$
exists. 
\end{proof}
\subsection{Well-definedness of $\mathbf{Ric}$ for admissible forward semigroup walks}
\label{subsubsec:olric-semigrp}
\par Suppose $\upmu_z^\epsilon$ is a semigroup walk. To such a semigroup walk $\upmu_z^\epsilon$, we associate its $\sf 1$-jet, $^1\upmu_z^{\varepsilon}$, given by
$
^1\upmu_z^{\varepsilon}
=
\delta_z+\varepsilon\mathcal{L}\delta_z
$.
 We have $^1\upmu_z^{\varepsilon}\in \mathscr{P}(\G)$ if and only if the generator $\Lo$ is weakly of divergence type (see \hyperref[defn:div-type]{Definition~\ref{defn:div-type}}), yet in general, the $1$-jet may have an infinite first moment; see \hyperref[subsubsec:inf-gen]{Section~\ref{subsubsec:inf-gen}}.
\subsubsection*{\small \textsf{\textbf{Proof of \hyperref[thm:main-1]{Main Theorem~\ref{thm:main-1}}}}}
\hfill
\par \textbf{\small Claim.}\label{claim:1} For all sufficiently small $\varepsilon>0$, all $z\in \G$ and for all natural cones $\mathsf{S}$ (see \hyperref[defn:natural-cone]{Definition~\ref{defn:natural-cone}}), we have
\begin{equation*}
	\dist(z,\cdot) \in \mathsf{Dom}\left(\widetilde{^\mathsf{S}\T}_{\varepsilon}\right).
\end{equation*}
\par \textbf{\small Proof of \hyperref[claim:1]{Claim}.}  
Let
\begin{equation*}
	g\in \mathsf{S}\cap \B, \quad \text{with} \quad 0 \le g \le \dist(z,\cdot),
\end{equation*}
this in particular, $g$ is bounded. For sufficiently large $N$, the function
$
f_N := N \wedge \dist(z,\cdot)
$
belongs to $\cb \cap \Lip_1(\G)$ and we have $g \le f_N$. 

Thus, by the positivity of $\T_\eps$, in the definition of
$
	\widetilde{^\mathsf{S}\T_{\varepsilon}}(\dist(z,\cdot))
$
we can take the supremum over
\begin{equation*}
	g \in \mathsf{S}\cap \B \cap \Lip(1), \quad \text{with} \quad 0 \le g \le \dist(z,\cdot).
\end{equation*}
\par Consider such a $g$. By the rough quasi-symmetry assumption, we can write
\begin{align*}
	\Big(\T_{\eps}(g) - g - \eps\mathcal{L}(g)\Big)(z) &= \sum_{w\in \G} \left( \T_{\eps} - \ident - \eps\mathcal{L} \right)\left(g\right)(w) \delta_z(w) \\
	& \le \Theta \left|\sum_{w\in \G} g(w) \left(\T_{\varepsilon}(\delta_z) - (\delta_z + \varepsilon\mathcal{L}\delta_z)\right)(w)\right|\\
	& \le \Uptheta \; \left|\Was_1\left(\T_\epsilon(\delta_z), \delta_z + \epsilon \mathcal{L}\delta_z\right)\right|.
\end{align*}
Hence, for sufficiently small $\epsilon >0$, we have
\begin{align*}
	\sup\limits_{g} \T_{\varepsilon}(g)(z) &\le \Uptheta \left|\Was_1\left(\T_{\varepsilon_0}(\delta_z), \delta_z  + \varepsilon \mathcal{L}\delta_z\right)\right| \\&\phantom{saj} + \sup\limits_{g} g(z) + \varepsilon\sup\limits_{g} \mathcal{L}(g)(z)\\
	&< \infty,
\end{align*}
since by the fact that $\mathcal L$ is comparable to multiplication, we have
\begin{equation*}
	\sup\limits_{g} \mathcal{L}(g)(z) \le q(z) g(z) < q(z) < \infty.
\end{equation*}
Hence, we have shown that
\begin{equation*}
	\widetilde{^\mathsf{S}\T_{\varepsilon}}\left(\dist(z,\cdot) \right)(z) < \infty.
\end{equation*}
\hfill \scalebox{0.8}{$\blacksquare$}
\begin{enumerate}
	\item[] \textsf{\textbf{\small Proof of the forward case:}} Now, to prove the case of admissible forward walks, we proceed as follows. Using the Kantorovich--Rubinstein formula and the uniform rough quasi-symmetry assumption (see \hyperref[defn:quasi-sym]{Definition~\ref{defn:quasi-sym}}), we have
	\begin{align}\label{eq:was-estimate}
		\Was_1\left(\upmu_z^{\varepsilon}, {^1}\upmu_z^{\varepsilon} \right) &= \sup_{f \in \Lip(1)\cap \fs} \sum_{w\in \G} f(w) \left(\upmu_z^{\varepsilon}(w) - {^1\upmu}_z^{\varepsilon}(w)\right) \notag
		\\&
		= \sup_{f \in \Lip(1)\cap \fs} \sum_{w\in \G} f(w) \left(\T_{\varepsilon}(\delta_z) - (\delta_z + \varepsilon\mathcal{L}\delta_z)\right)(w)
		\\&
		\le \Uptheta \sup_{f \in \Lip(1) \cap \fs}  \left(\T_{\varepsilon}\left(f\right) - f - {\varepsilon}\mathcal{L}\left(f\right)\right)(z)\notag,
	\end{align}
	for some $\Uptheta \ge 1$ determined by the uniform rough quasi-symmetry assumption. So we need to find an estimate on the order of smallness of the RHS of \hyperref[eq:was-estimate]{\eqref{eq:was-estimate}}. 
	\par Let $f \in \Lip(1) \cap \fs$ be a test function. Since $\Lo$ is strongly local, there exists $S(z)$ such that
	\begin{equation*}
	\supp \left( \Lo \delta_{z}  \right) \subset \B_{S(z)}(z).
	\end{equation*}	
\par By the translation invariance of the Kantorovich--Rubinstein formula, we may replace $f$ by $\widetilde{f}$ given by
	\begin{equation*}
	\widetilde{f} = f - f(z) + S(z)
	\end{equation*}
and instead assume that 
	\begin{equation*}
	f\in \Lip(1) \cap \mathsf{C}_{\sf b}, \quad \text{and} \quad f\left(z\right) = S\left(z\right).
	\end{equation*}
We note that in the RHS of \hyperref[eq:was-estimate]{\eqref{eq:was-estimate}}, is also translation invariant since $\T_\eps$ is assumed to be formally stochastically complete; see \hyperref[subsubsec:inf-gen]{Section~\ref{subsubsec:inf-gen}}. 
 \par Consequently, we have
	$
		f > 0
	$
	on
	$
	\B_{S(z)}\left(z\right)
	$.
	Also we have
	$
		{}^1\upmu_{z}^\epsilon = 0
	$ on 
	$
	\G \smallsetminus \B_{S(z)}\left(z\right)
	$.
\par By the non-negativity of the semigroup $\T_\varepsilon$, $\T_{\varepsilon}(\delta_{z})$ is non-negative everywhere. Hence, we deduce that
\begin{equation*}
	\T_{\varepsilon}(\delta_{z}) - \left(\delta_{z} + \varepsilon\mathcal{L}\delta_{z}\right)\ge 0, \quad \text{on} \quad \G \smallsetminus \B_{S(z)}\left(z\right).
\end{equation*}
In particular, these facts imply that we can replace the test functions $f$ by $f^+ \in \Lip(1)$ without changing the optimal value in \hyperref[eq:was-estimate]{\eqref{eq:was-estimate}}. This is because doing so does not decrease the right-hand side of \hyperref[eq:was-estimate]{\eqref{eq:was-estimate}}: the summation over $\B_{S(z)}\left(z\right)$ remains unchanged, while the summation over
$
\G \smallsetminus \B_{S(z)}\left(z\right)
$
does not decrease. Henceforth, we may also assume that $f \ge 0$.
\par By Taylor's theorem with the Lagrange form of the remainder (see the discussion in \hyperref[subsubsec:semigroups]{Section~\ref{subsubsec:semigroups}}), we have
\begin{equation*}
	\Big(\T_{\varepsilon}\left(f\right) - f - \varepsilon\mathcal{L}\left(f\right)\Big)(z) = \nicefrac{{\varepsilon}^2}{2} \; \mathcal{L}^2\T_{\tau(f)} \left(f\right)(z),
\end{equation*}
for some
$
\tau(f) \in (0,\eps)
$.
\par To bound the term
$
	\mathcal{L}^2\T_{\tau(f)} \left(f \right)(z)
$,
we argue as follows.
From
\begin{equation*}
	0\le f \le S(z) + \dist(z,\cdot),
\end{equation*}
we deduce that
\begin{equation*}
	0 \le \T_\tau\left(f\right) = \widetilde{^\mathsf{S}\T_\tau}\left(f\right) \le \widetilde{^\mathsf{S}\T_\tau}\left(f(z) + \dist(z,\cdot)\right),
\end{equation*}
where we have used the monotonicity of the conical supremum; see \hyperref[subsubsec:conical-sup]{Section~\ref{subsubsec:conical-sup}}.
\par Since $\mathcal{L} \ge q$, we obtain
\begin{align*}
	\widetilde{^\mathsf{S}\T_\tau}\left(S(z) + \dist(z,\cdot)\right)(x) &\le \widetilde{^\mathsf{S}\T_{\eps}}\left(S(z) + \dist(z,\cdot)\right)(x) e^{(\eps -\tau)q(x)}\\
	&\le \widetilde{^\mathsf{S}\T_{\eps}}\left(S(z) + \dist(z,\cdot)\right)(x) e^{q(x)};
\end{align*}
see \hyperref[subsubsec:consup-comp]{Section~\ref{subsubsec:consup-comp}} and \hyperref[defn:rough-comp-mult]{Definition~\ref{defn:rough-comp-mult}}.
\par By the semi-locality of $\mathcal{L}$, there exists a radius $a>0$ such that
\begin{equation*}
	\mathcal{L}^2\T_{\tau} \left(f \right)\left(z\right) = \mathcal{L}^2 \Big(\chi_{B_a\left(z\right)} \T_{\tau} \left(f \right) \Big)\left(z\right).
\end{equation*}
\par Since $\mathcal{L}$ is comparable to multiplication, we deduce
\begin{align}\label{eq:unif-bd-1}
	\left|\mathcal{L}^2\left(\chi_{B_a(z)}\T_{\tau} \left(f \right)(z)\right)\right| &\le C(a,z)\left\|\chi_{B_a(z)}\T_{\tau} \left(f \right)\right\|_{\sup}\notag \\
	&\le C(a,z)\left\|\chi_{B_a(z)}\; \widetilde{^\mathsf{S}\T_\tau}\left(S(z) + \dist(z,\cdot)\right)\right\|_{\sup}.
\end{align}
In particular,
\begin{align}\label{eq:unif-bd-2}
	&\left\|\chi_{B_a(z)}\; \widetilde{^\mathsf{S}\T_\tau}\left(S(z) + \dist(z,\cdot)\right)\right\|_{\sup} \notag \\
	&\le C(a,z) \left\|\chi_{B_a(z)}\; \widetilde{^\mathsf{S}\T_{\eps}}\left(S(z) + \dist(z,\cdot)\right)\right\|_{\sup}.
\end{align}
Upon combining \hyperref[eq:unif-bd-1]{\eqref{eq:unif-bd-1}} and \hyperref[eq:unif-bd-2]{\eqref{eq:unif-bd-2}}, we deduce the desired uniform bound.
\par Thus, we have shown that
\begin{equation}\label{eq:second-order}
	\Big(\T_{\varepsilon}\left(f\right) - f - t\mathcal{L}\left(f\right)\Big)(z) \le C \varepsilon^2,
\end{equation}
where the constant $C$ depends only on $z$ and the semi-locality scales of $\mathcal L$. Based on \hyperref[eq:was-estimate]{\eqref{eq:was-estimate}}, we obtain
\begin{equation*}
	\Was_1\left(\upmu_z^{\varepsilon}, {^1}\upmu_z^{\varepsilon} \right) \le \Theta C \epsilon^2.
\end{equation*}
\par Therefore, the preceding argument can be repeated, and the final constant $C$ appearing in \hyperref[eq:second-order]{\eqref{eq:second-order}} does not increase as $\eps\downarrow 0$. Hence, the $1$-jet approximation 
\begin{equation*}
	\Was_1\left(\upmu_z^\varepsilon,^1\upmu_z^\varepsilon\right)
	=
	\mathcal{O}\left(\varepsilon^2\right)
	\qquad\text{as}\qquad
	\varepsilon\downarrow0,
\end{equation*}	
follows for admissible forward semigroup walks. 
\bigskip
\item [] \textsf{\textbf{\small Proof of the backward case:}}
\hfill\newline
To prove the case of backward admissible walks, we proceed as follows. 
Using the Kantorovich--Rubinstein formula with finite support, we have
\begin{align*}
	\Was_1\left(\upmu_z^{\varepsilon_0}, {^1}\upmu_z^{\varepsilon_0} \right) &= \sup_{f \in \Lip(1) \cap \fs} \sum_{w\in \G} f(w) \left(\upmu_z^{\varepsilon_0}(w) - {^1\upmu}_z^{\varepsilon_0}(w)\right) \notag
	\\&
	= \sup_{f \in \Lip(1) \cap \fs} \sum_{w\in \G} f(w) \left(\T_{\varepsilon_0}(\delta_w) - (\delta_w + \varepsilon_0\mathcal{L}\delta_w)\right)(z)
	\\&= \sup_{f \in \Lip(1) \cap \fs}  \left(\T_{\varepsilon_0}\left(f\right) - f - {\varepsilon_0}\mathcal{L}\left(f\right)\right)(z) \;\notag
\end{align*}
\par By the translation invariance of the Kantorovich--Rubinstein formula, the left-hand side does not change when a constant is added to $f$. The same is true of the right-hand side, since $\Lo$ is translation invariant and $\T_\eps$ is formally stochastically complete; see \hyperref[defn:formal-stoch-compl]{Definition~\ref{defn:formal-stoch-compl}}. 
\par So we again need to find an estimate on the order of smallness in $\eps$ for
\begin{equation*}
\left(\T_{\varepsilon}\left(f\right) - f - {\varepsilon}\mathcal{L}\left(f\right)\right)(z)
\end{equation*}
when $f$ satisfies
$
f \in \Lip_1(\G)
$
and $f(z) = R_0$
and where we have
\begin{equation*}
	{}^1\upmu_{z}^\epsilon(w) = \delta_w + \eps \Lo \delta_w(z) = 0, \quad \text{for all} \quad w\in\G \smallsetminus \B_{R_0}\left(z\right),
\end{equation*}
by the uniform semi-locality at scale $R_0$. 
\par From here, obtaining the estimate then follows verbatim as in the case for the admissible forward semigroup walks. 
\end{enumerate}
\par Finally, to show that the Ollivier--Ricci curvature coincides with that of the corresponding $1$-jet walk, we note that, since $\Was_1$ metrizes $\mathscr{P}_1(\G)$ and since
$
{}^1\upmu_z^\varepsilon \in \mathscr{P}_1(\G)
$
, we have
$
	\Was_1\left(\upmu_z^\varepsilon,{}^1\upmu_z^\varepsilon \right)< \infty
$;
Thus, by the argument above, we have
\begin{equation*}
	\Was_1\left(\upmu_z^\varepsilon,{}^1\upmu_z^\varepsilon\right)
	=
	\mathcal{O}\left(\varepsilon^2\right)
	\qquad\text{as}\qquad
	\varepsilon\downarrow0.
\end{equation*}
Now
$
	\Ric(x,y) = {}^1\Ric(x,y)
$
follows by upon applying the triangle inequality. Recall ${}^1\Ric$ is well-defined according to \hyperref[lem:wd-affine]{Lemma~\ref{lem:wd-affine}}.
\newline\phantom{}\hfill\scalebox{1.2}{\qed}
\begin{example}\phantom{}\hspace{-6pt}\label{ex:1}
Suppose $\G$ is a locally finite graph and let $\dist$ denote the combinatorial distance. 
\par Suppose $\G$ is stochastically complete w.r.t. heat flow generated by the drift-less ($\mathdutchcal{g}\equiv 0$) discrete Laplacian $	\Delta_{\omega, 0,\m}$ which is of the form (see \hyperref[subsec:rudiments]{Section~\ref{subsec:rudiments}})
\begin{equation*}
	\Delta_{\omega, 0,\m} f (x):= \m(x)^{-1}\sum_{x \sim y} \left( f(y) - f(x)  \right) \omega_{xy}.
\end{equation*}
By the work in~\cite{KLW}, we know $\Delta_{\omega, 0,\m}$ generates a contractive Markovian semigroup  
$
	\T_t: \cb \to \cb
$.
Consequently,
\begin{equation*}
\upmu^\eps_z(w) := \T_\eps(\delta_w)(z) = p(\eps,x,dw)>0
\end{equation*}
is an admissible backward integral semigroup walk.
\par It is easy to see that
\begin{equation*}
\left| \Delta_{\omega,0, \m} f(x)\right| \le q(x)f(x) :=\frac{1}{\m(x)} \max \left\{\sum\limits_{y\sim x} \omega(x,y),\, \max_{y\sim x}\omega(x,y)  \right\}
\end{equation*}
where
\begin{equation*}
q(x) :=\frac{1}{\m(x)} \max \left\{\sum\limits_{y\sim x} \omega(x,y), \max_{y\sim x}\omega(x,y)  \right\}
\end{equation*}
thus by local finiteness, we deduce $\Delta_{\omega, 0 , \mathdutchcal{n}}$ is a uniformly semi-local operator that is comparable to multiplication. Its domain contains all functions in $\cb$ thus the smeigorup is clearly $\infty$-regular. 
\par To see that $\Delta_{\omega,0, \m}$ is comparable to multiplication in the sense of \hyperref[defn:rough-comp-mult]{Definition~\ref{defn:rough-comp-mult}}, let $|f| \le M$, then one clearly has
\[
\left|\Delta_{\omega,0, \m} f(x)  \right|\le \frac{2 \sum_{y\sim x} \omega(x,y)}{\m(x)} M.
\]
\par Therefore, upon applying \hyperref[thm:main-1]{Main Theorem~\ref{thm:main-1}}, we deduce that $\Ric$ is well-defined and coincides with ${}^1\Ric$. 
\end{example}
\section{Variational formulation for the Ollivier--Ricci curvature}\label{sec:var}
\subsection{Normal forms of general random walks}\label{subsec:normal-form}
It would be useful to obtain an operator-theoretic expression for a general class of random walks under minimal assumptions.
\subsubsection{\small \textsf{\textbf{Standard backward normal form}}}
Let $\upmu_z^\varepsilon$ be a random walk. One associates to it the (non-log-linear in $\varepsilon$) averaging operators
\begin{equation*}
\mathscr{A}^\varepsilon f(z) := \int_{\Upomega_{z}} f(w)\,d\upmu_z^\varepsilon(w);
\end{equation*}
thus,
$
\mathscr{A}^\varepsilon \delta_w(z) = \upmu_z^\varepsilon(w)
$.
One also associates the integrated random-walk Laplacians
$
\mathscr{L}^\varepsilon = \mathscr{A}^\varepsilon - \ident
$;
see~\cite{Ol}.
\par Therefore,
\begin{equation*}
\mathscr{L}^\varepsilon f(z)
= \sum_{w\in \G} \left(f(w)-f(z)\right)\upmu_z^\varepsilon(w)
= \sum_{w\in \G\smallsetminus\{z\}} \left(f(w)-f(z)\right)\upmu_z^\varepsilon(w),
\end{equation*}
is well-defined on a domain containing $\mathsf{C}_{\sf b}$ and $\Lip(\G)$, owing to the finite first-moment assumption on the random walks.
\par When $\upmu_z^\epsilon$ is vertex-wise differentiable at $\varepsilon=0$, one has the asymptotic development
\begin{align}\label{eq:gen-lap-diff}
	\mathscr{L}f(z)
	:= \lim_{\varepsilon \downarrow 0}
	\nicefrac{1}{\varepsilon}
	\left(\mathscr{A}^\varepsilon-\mathbb{I}\right)f(z)
	= \nicefrac{d}{d\varepsilon}\restr_{\varepsilon=0}
	\mathscr{L}^\varepsilon f(z)
	= \sum_{\G\smallsetminus{z}}
	\left(f(w)-f(z)\right)
	\;\nicefrac{d}{d\varepsilon}\restr_{\varepsilon=0}
	\upmu_z^\varepsilon(w),
\end{align}
which is well-defined for $f\in\fs$.
If, furthermore, $\upmu_z^\epsilon$ is $\mathsf{TR}^{2,1}$-regular, then one has the asymptotic development
\begin{equation*}
\upmu_z^\varepsilon(w)
= \delta_z
+ \epsilon\,\nicefrac{d}{d\varepsilon}\restr_{\varepsilon=0}
\upmu_z^\varepsilon
+ \mathcal{O}\left(\epsilon^2\right),
\end{equation*}
where the error term, multiplied by $\varepsilon^{-1}$, is uniformly summable. Thus, \hyperref[eq:gen-lap-diff]{\eqref{eq:gen-lap-diff}} holds on a domain containing $\mathsf{C}_{\sf b}$. \emph{Notice that $\mathsf{TR}_0^{2,1}$-regularity guarantees that the limit can be taken inside the summation when $f\in\mathsf{C}_b$.} In a similar fashion, if the walk is in $^\dist\mathsf{TR}_0^{2,1}$, then the domain of $\mathcal{L}$ will also include $\Lip_1(\G)$ (see~\hyperref[subsubsec:imp-func-spaces]{Section~\ref{subsubsec:imp-func-spaces}} for the definitions of these function spaces).
\par In particular, by the uniqueness of asymptotic developments~\cite[Theorem 5.5]{Cole}, we obtain
\begin{equation*}
\upmu_z^\varepsilon(w)
= \delta_w(z)
+ \varepsilon\,\mathscr{L}\delta_w(z)
+ \mathcal{O}\left(\epsilon^2\right).
\end{equation*}
We call this the \emph{normal backward form} of $\upmu_z^\varepsilon$, and we refer to $\mathscr{L}$ as the \emph{backward initial generator}.
\par \emph{Furthermore, $1\in {\sf Dom}(\mathscr{L})$ and $\mathscr{L}1=0$.}
Notice that, if $\mathscr{L}$ is continuous with respect to the point-open topology (the topology of pointwise convergence), then the latter condition is equivalent to the first-order term $\varepsilon\mathscr{L}\delta_w(z)$ being a zero-mass signed measure.
\par If
\begin{equation*}
\nicefrac{d}{d\varepsilon}\restr_{\varepsilon=0}
\upmu_z^\varepsilon(w)
=
\nicefrac{d}{d\varepsilon}\restr_{\varepsilon=0}
\upmu_w^\varepsilon(z),
\end{equation*}
holds for all $z,w\in\G$, then $\mathscr{L}$ is roughly symmetric with respect to the counting measure. In particular,
\begin{equation*}
\upmu_z^\varepsilon(w)
= \delta_z(w)
+ \varepsilon\,\mathscr{L}\delta_z(w)
+ \mathcal{O}(\epsilon^2).
\end{equation*}
\par For a semigroup walk, $\mathscr{L}$ is called the \emph{backward generator}. If the walk is $2$-regular and symmetric, the backward generator $\mathscr{L}$ and the infinitesimal generator of the semigroup coincide on
$
\fs \cap \Dom(\Lo)
$,
by the uniqueness of the asymptotic developments. See also the \hyperref[ex:1]{Example~\ref{ex:1}} below.
\begin{example}\phantom{}\hspace{-6pt}\label{ex:1}
	Suppose that $\Lo$ is an operator that is weakly in divergence form and that
	$
		\upmu_z^\eps = \delta_z + \eps \Lo \delta_z
	$
	defines a time-affine random walk. Taking $f = \delta_{z_0}$, we obtain
	\begin{equation*}
		\mathscr{A}^\varepsilon (\delta_{z_0})(z)
		:= \int_{\G} \delta_{z_0}(w)\,d\upmu_z^\varepsilon(w) = (\ident + \eps \Lo)\delta_z(z_0).
	\end{equation*}
	Consequently,
	\begin{equation}\label{eq:time-affine-form-f}
		\mathscr{L}^\varepsilon \delta_{z_0}(z)
		= \eps \Lo \delta_z(z_0),
		\quad \text{and} \quad
		\mathscr{L} \delta_{z_0}(z)
		= \Lo \delta_z(z_0).
	\end{equation}
\par Similarly, suppose $\Lo$ is an operator that is translation invriant and that
\begin{equation*}
	\upmu_z^\eps(w) = \left(\delta_w + \eps \Lo \delta_w\right)(z)
\end{equation*}
defines a time-affine random walk. Taking $f = \delta_{z_0}$, we obtain
\begin{equation*}
	\mathscr{A}^\varepsilon (\delta_{z_0})(z)
	:= \int_{\G} \delta_{z_0}(w)\,d\upmu_z^\varepsilon(w) = (\ident + \eps \Lo)\delta_{z_0}(z).
\end{equation*}
Consequently,
\begin{equation}\label{eq:time-affine-form-b}
	\mathscr{L}^\varepsilon \delta_{z_0}(z)
	= \eps \Lo \delta_{z_0}(z),
	\quad \text{and} \quad
	\mathscr{L} \delta_{z_0}(z)
	= \Lo \delta_{z_0}(z).
\end{equation}
\end{example}
\subsection{Variational formulation}\label{subsec:var-form}
\par An interesting variational (or limit-free) formulation of the Ollivier--Ricci curvature of time-affine lazy random walks of the (backward normal) form
\begin{equation*}
\upbeta_z^\eps(w) := \left(\ident + \eps \Delta\right)\delta_w(z),
\end{equation*}
was first observed in~\cite{MW}.
\par Here, we establish a variational formulation for a general class of random walks. The case of admissible semigroup walks will follow as a special case.
\par We first make a simple observation. For any semi-local at a controlled scale operator $\mathcal{A}$, a direct consequence of the Lipschitz tapering in \hyperref[lem:Lip-taper]{Lemma~\ref{lem:Lip-taper}} is that
\begin{equation*}
\inf_{\substack{ \\ f \in \Lip(1)\\ \nabla_{\vv{xy}} f = 1 \\ f(x)=0 }}
\nabla_{\vv{yx}} \mathcal{A}(f)
=
\inf_{\substack{ \\ f \in \Lip(1) \cap \fs\\ \nabla_{\vv{xy}} f = 1 \\ f(x)=0 }}
\nabla_{\vv{yx}} \mathcal{A}(f),
\end{equation*}
for all $x,y \in \G$. In fact, it suffices to take the infimum over functions $f$ satisfying
\begin{equation*}
\supp(f) \subset B_{2\zeta}(x), \quad \zeta = \Uplambda + s(\Uplambda),
\end{equation*}
where $\Uplambda$ is chosen so that
$
\Uplambda > R(x)
$.
\par Consequently, we can take the infimum over 
\begin{equation*}
f \in \Lip(1) \cap \cb,\quad \nabla_{\vv{xy}} f = 1, \quad f(x)=0.
\end{equation*}
\par In what follows, $\upmu_z^\varepsilon$ is a given random walk with averaging operator $\mathscr{A}^\epsilon$, integrated Laplacians $\mathscr{L}^\epsilon$, and backward initial generator $\mathscr{L}$; see \hyperref[subsec:normal-form]{Section~\ref{subsec:normal-form}} for details. 
\subsubsection*{\small \textsf{\textbf{Proof of \hyperref[thm:main-2]{Main Theorem~\ref{thm:main-2}}}}}
	\par By Kantorovich--Rubinstein duality, one has
	\begin{equation*}
		\Was_1(\upmu_x^\varepsilon,\upmu_y^\varepsilon)
		=
		\sup_{\substack{\ f \in \Lip(1)}}
		\sum f(z)\left(\upmu_x^\varepsilon(z)-\upmu_y^\varepsilon(z)\right);
	\end{equation*}
	hence,
	\begin{align*}
		\Was_1(\upmu_x^\varepsilon,\upmu_y^\varepsilon)
		&=
		\sup_{f \in \Lip(1)}
		\left[
		\left(f(x)-f(y)\right)
		+
		\Big(
		\mathscr{L}^\varepsilon(f)(x)
		-
		\mathscr{L}^\varepsilon(f)(y)
		\Big)
		\right] \\
		&=
		\dist(x,y)
		\sup_{f \in \Lip(1)}
		\Big(
		\nabla_{\vv{yx}} f
		+
		\nabla_{\vv{yx}}\mathscr{L}^\varepsilon(f)
		\Big);
	\end{align*}
	and, as a result,
	\begin{align*}
		\Ric_\varepsilon(x,y)
		&=
		1-\frac{\Was_1(\upmu_x^\varepsilon,\upmu_y^\varepsilon)}
		{\dist(x,y)}\\
		&=
		\inf_{f \in \Lip(1)}
		\Big(
		(1-\nabla_{\vv{xy}}f)
		+
		\nabla_{\vv{yx}}\mathscr{L}^\varepsilon(f)
		\Big).
	\end{align*}
\par Adding the constraint $\nabla_{\vv{xy}}f=1$, we obtain
	\begin{align}\label{eq:ric-eps-bound}
		\nicefrac{1}{\varepsilon}\Ric_\varepsilon
		&=
		\inf_{f \in \Lip(1)}
		\nicefrac{1}{\varepsilon}
		\Big(
		(1-\nabla_{\vv{xy}}f)
		+
		\nabla_{\vv{yx}}\mathscr{L}^\varepsilon(f)
		\Big)\\
		&\le
		\inf_{\substack{f \in \Lip(1)\ \nabla_{\vv{xy}}f=1}}
		\nicefrac{1}{\varepsilon}
		\nabla_{\vv{yx}}\mathscr{L}^\varepsilon(f)\notag\\
		&=
		\inf_{\substack{f \in \Lip(1)\ \nabla_{\vv{xy}}f=1}}
		\left(
		\nabla_{\vv{yx}}\mathscr{L}(f)+\smallO_\eps(1)
		\right)\notag\\
		&=
		\inf_{\substack{f \in \Lip(1)\ \nabla_{\vv{xy}}f=1}}
		\nabla_{\vv{yx}}\mathscr{L}(f)
		+
		\smallO_\eps(1).
		\notag
	\end{align}
	where the last equality follows from assumption~\hyperref[item:lff-1]{(a)}. Hence,
	\begin{equation*}
	\Ric^+(x,y)
	=
	\varlimsup_{\epsilon\to0}
	\nicefrac{1}{\varepsilon}\Ric_\varepsilon
	\le
	\inf_{\substack{f \in \Lip(1)\\ \nabla_{\vv{xy}}f=1}}
	\nabla_{\vv{yx}}\mathscr{L}(f).
	\end{equation*}
	\par To prove the second claim, suppose that $\mathcal{L}$ is semi-local and that $\Ric^+(x,y)$ is finite. For every $\epsilon>0$, suppose that \hyperref[eq:ric-eps-bound]{\eqref{eq:ric-eps-bound}} attains its minimum at $f_\varepsilon$. We note that existence of a minimizer is guaranteed by the Kantorovich--Rubinstein duality theorem; see, e.g.,~\cite[Theorem 5.10]{Vil}.
\par The problem is invariant under translations of $f$, so we may assume that $f_\varepsilon(x)=0$. By the argument at the beginning of this section, we may assume that all $f_\varepsilon$ are supported in $B_{4R(x)}(x)$.
\par By a standard diagonal argument, there exists a subsequence $\varepsilon_k$ such that $f_{\varepsilon_k}$ converges pointwise to some function $f_0$.
\par {\small \bf Claim.}
	$
	\nabla_{\vv{xy}}f_0=1.
	$
\par {\small \bf Proof of the claim.} Since the functions $f_{\varepsilon_k}$ are all $1$-Lipschitz, satisfy $f_{\varepsilon_k}(x)=0$, and satisfy
$
	\supp(f_{\varepsilon_k})\subset B_{4R(x)}(x)
$,
	we deduce that $\|f_{\varepsilon_k}\|_{\sup}$ is uniformly bounded.
	\par By the uniform asymptotics assumed in \hyperref[item:lff-1]{(a)} (see \hyperref[subsubsec:unif-asymp]{Section~\ref{subsubsec:unif-asymp}}), we then obtain
	\begin{equation*}
	\mathscr{L}^{\varepsilon_k}
	\left(f_{\varepsilon_k}\right)
	=
	\varepsilon_k\mathscr{L}f_{\varepsilon_k}
	+
	\mathcal{O}(\epsilon_k^2),
	\end{equation*}
	where the error term is uniform in $k$. In particular,
	\begin{equation*}
	\nicefrac{1}{\epsilon_k}
	\nabla_{\vv{yx}}
	\mathscr{L}^{\varepsilon_k}
	\left(f_{\varepsilon_k}\right)
	=
	\nabla_{\vv{yx}}
	\nicefrac{1}{\epsilon_k}
	\mathscr{L}^{\varepsilon_k}
	\left(f_{\epsilon_k}\right)
	\end{equation*}
	is bounded as $\epsilon_k\downarrow0$, and, by the point-open continuity of $\mathcal{L}$, converges to
	$
	\nabla_{\vv{yx}}\mathscr{L}(f_0)
	$;
	see \hyperref[rem:uniform-diff]{Remark~\ref{rem:uniform-diff}}.
	\par Consequently, in \hyperref[eq:ric-eps-bound]{\eqref{eq:ric-eps-bound}}, we must have
	\begin{equation*}
	1-\nabla_{\vv{xy}}f_0
	=
	\lim_{\epsilon_k\downarrow0}
	\left(
	1-\nabla_{\vv{xy}}f_{\epsilon_k}
	\right)
	=
	0,
	\end{equation*}
since otherwise $\Ric^+(x,y)$ would not be finite. \hfill \scalebox{0.8}{$\blacksquare$}
	\par Since $f_0$ is $1$-Lipschitz and $\nabla_{\vv{xy}}f_0=1$, we can use it as a test function. We therefore deduce
	\begin{align*}
		\varlimsup_{\epsilon\to0}
		\nicefrac{1}{\varepsilon}\Ric_\varepsilon
		\ge
		\lim_{k\to\infty}
		\nicefrac{1}{\varepsilon_k}
		\nabla_{\vv{yx}}
		\mathscr{L}^{\varepsilon_k}
		\left(f_{\epsilon_k}\right)
		=
		\nabla_{\vv{yx}}\mathscr{L}f_0
		\ge
		\inf_{\substack{f \in \Lip(1)\cap\fs\\
				\nabla_{\vv{xy}}f=1}}
		\nabla_{\vv{yx}}\mathscr{L}(f).
	\end{align*}
\hfill\scalebox{1.2}{\qed}
\subsubsection*{\small \textsf{\textbf{Proof of \hyperref[cor:main-2]{Main Corollary~\ref{cor:main-2}}}}}
\par Based on \hyperref[thm:main-1]{Main Theorem~\ref{thm:main-1}}, we only need to establish the variational formulation for ${}^1\Ric$.
\par According to \hyperref[eq:time-affine-form-f]{\eqref{eq:time-affine-form-f}} and \hyperref[eq:time-affine-form-b]{\eqref{eq:time-affine-form-b}}, $\mathscr L^\eps$ is linear in $\eps$ and therefore obviously has uniform second-order asymptotics.
\par For admissible forward semigroup walks, one sees that $\mathscr{L}$ is continuous with respect to pointwise convergence of finitely supported functions as follows. Suppose that $f\in \fs$. We have
		\begin{equation*}
			\mathscr{L}g(z) = \sum_{w} g(w) \Lo \delta_z(w)
			= \left<f, \Lo \delta_z\right>,
		\end{equation*}
		where the last equality follows from the semi-locality of $\Lo$. It is clear that the right-hand side is continuous with respect to pointwise convergence in $\fs$.
\par Therefore, in the case of admissible forward semigroup walks, all the hypotheses of \hyperref[thm:main-2]{Main Theorem~\ref{thm:main-2}} are satisfied. Thus, from
		\begin{align*}
			\dist(x,y)\nabla_{\vv{yx}}\mathscr{L}(g)
			&= \left<f, \Lo \left(\delta_x - \delta_y \right)\right>,
		\end{align*}
		and upon taking the infimum over
		$
		f \in \Lip(1) \cap \fs
		$ with $\nabla_{\vv{xy}} f = 1$,
		we conclude the proof. 
\par For admissible backward semigroup walks, $f\in \fs$, one has
\begin{equation*}
	\mathscr{L}f(z) = \sum_{w} f(w) \Lo \delta_w(z)
	= \Lo f(z)
\end{equation*}
thus, the continuity under the pointwise convergence in $\fs$ is equivalent to the same property for the infinitesimal generator $\Lo$. Therefore, once again, we can apply \hyperref[thm:main-2]{Main Theorem~\ref{thm:main-2}}.
\par Based on the discussion at the beginning of \hyperref[subsec:var-form]{Section~\ref{subsec:var-form}} and the assumption 
$
\Lip(\G) \cap \mathsf{C}_{\sf b} \subseteq \Dom(\Lo)
$,
it is clear that in the generator representation we can replace the space $\Lip_1(\G)\cap \fs$ in
\begin{equation*}
	\Ric(x,y)
	=
	\inf_{\substack{f \in \Lip(1) \cap \fs\\
			\nabla_{\vv{xy}} f = 1 }}
	\nabla_{\vv{yx}} \mathcal{L}(f)
\end{equation*}
with spaces 
$
\Lip_1(\G) \cap \Dom(\Lo) 
$
or
$
\Lip_1(\G) \cap \cb
$.
\hfill\scalebox{1.2}{\qed}
\end{introthm}
\begin{example}\phantom{}\hspace{-10pt}(General graph Laplacian)
Let $\left(\G,\dist\right)$ denote a locally finite graph equipped with the combinatorial distance. Consider the backward Markovian walk generated by the Laplacian,
$
\Delta_{\omega, 0, \m}
$,
presented in \hyperref[ex:1]{Example~\ref{ex:1}}. Clearly by local-finiteness, the drift-less Laplacian $\Delta_{\omega, \m}$ is continuous w.r.t. pointwise convergence. Therefore, upon applying \hyperref[cor:main-2]{Main Corollary~\ref{cor:main-2}} -- \hyperref[item:GR2]{(\ref{item:GR2})}, we deduce that
\begin{equation*}
	\Ric(x,y)
	=
	\inf_{\substack{f \in \Lip(1) \cap \fs\\
			\nabla_{\vv{xy}} f = 1 }}
	\nabla_{\vv{yx}} \Delta_{\omega,0, \m}(f)
	=
	\inf_{\substack{f \in \Lip(1)\\
			\nabla_{\vv{xy}} f = 1 }}
	\nabla_{\vv{yx}} \Delta_{\omega,0, \m}(f).
\end{equation*}
Hence the limit-free formulation presented \cite{MW} is included in our framework.  
\end{example}
\section{Operator-theoretic Ricci curvature}\label{sec:op-theo-ric}
\begin{definition}[Weighted comparison principles]\label{defn:comp-principle}
	An operator $\mathcal{A}$ is said to satisfy
	\begin{itemize}
		\item \label{comp-item-1} a \emph{maximum principle with range $R\geq0$} if there exists a function
		$
			C:\G^2\times\R\to\R
		$
		such that, for any
		$
			f\in \mathsf{Dom}(\mathcal{A}) \cap \Lip_1(\G)
		$,
		and for all $z\in \B_R(x_0)$, we have
		\begin{equation*}
			\text{$f$ has a global maximum at $x_0$}
			\Longrightarrow
			\mathcal{A}(f)(z)
			\leq
			C(x_0,z,f(x_0)).
		\end{equation*}
		\item \label{comp-item-2} a \emph{two-sided comparison principle with range $R$} if there exists a non-negative function
		$
			D:\G^2\times\R\to\R
		$
		such that, for
		$
			f\in \mathsf{Dom}(\mathcal{A}) \cap \Lip_1(\G)
		$,
		and for all  $z\in\B_R(x_0)$, we have
		\begin{equation*}
			\text{$f$ has a global extremum at $x_0$}
			\Longrightarrow
			\left|\mathcal{A}(f)(z)\right|
			\leq
			D(x_0,z,f(x_0)).
		\end{equation*}
	\end{itemize}
	\par \emph{When the range is zero, we simply say that we have a maximum principle (respectively, a two-sided comparison principle).}	
\end{definition}
\par Suppose $(X,\dist_X)$ is a complete metric space. Let
\begin{equation*}
\nabla_{\vv{xy}} f := \dist(x,y)^{-1}\left(f(y)-f(x)\right), \quad x\neq y.
\end{equation*}
\par Let $\B$ be a Banach space of functions on $X$, and let
$
\mathcal A:\Dom(\mathcal A)\subseteq\B\to\B
$
be a linear operator.
\par Inspired by the variational characterization of the Ollivier--Ricci curvature, we propose, associated with the operator $\mathcal A$, the following three notions of operator-theoretic Ricci curvature with values in extended real numbers.
\begin{definition}
	\hfill
	\begin{itemize}
		\item The operator-theoretic Ollivier--Ricci curvature is defined as
		\begin{equation*}
		\mathcal Ric_{\mathcal{A}}(x,y)
		:=
		\inf_{\substack{f\in\Lip(1)\cap\mathsf{Dom}(\mathcal A)\\
				\nabla_{\vv{xy}}f=1}}
		\nabla_{\vv{yx}}\mathcal Af,
		\quad x\neq y.
		\end{equation*}
		\item The restricted operator-theoretic Ollivier--Ricci curvatures ${}^{\sf c}\Ric_{\mathcal{A}}$ and ${}^{\sf b}\Ric_{\mathcal{A}}$ are defined by
		\begin{equation*}
		{}^{\sf c}\Ric_{\mathcal{A}}(x,y)
		:=
		\inf_{\substack{
				f\in\Lip(1)\cap\mathcal{C}_{\sf cs}(\G)\cap\mathsf{Dom}(\mathcal A)\\
				\nabla_{\vv{xy}}f=1}}
		\nabla_{\vv{yx}}\mathcal Af,
		\quad x\neq y.
		\end{equation*}
		and
		\begin{equation*}
			{}^{\sf b}\Ric_{\mathcal{A}}(x,y)
			:=
			\inf_{\substack{
					f\in\Lip(1)\cap\mathcal{C}_{\sf b}(\G)\cap\mathsf{Dom}(\mathcal A)\\
					\nabla_{\vv{xy}}f=1}}
			\nabla_{\vv{yx}}\mathcal Af,
			\quad x\neq y,
		\end{equation*}
		respectively.
	\end{itemize}
\end{definition}
\par For a semigroup walk $\upmu_z^\varepsilon:=e^{\varepsilon\mathcal L}\updelta_z$, the Ollivier--Ricci curvature might not be well-defined. However, the operator-theoretic Ollivier--Ricci curvature is well-defined, although it may take the values $\pm\infty$.
\par The following results provide sufficient conditions for the well-definedness of the operator-theoretic Ricci curvature.
\begin{theorem}
	Suppose $\mathcal{L}$ satisfies the following properties:
	\begin{enumerate}
		\item [(OR1)] $\mathcal{L}$ is translation-invariant;
		\smallskip
		\item [(OR2)] $\mathcal{L}$ satisfies a two-sided comparison principle (see \hyperref[defn:comp-principle]{Definition~\ref{defn:comp-principle}}).
	\end{enumerate}
	Then $\mathcal Ric_{\mathcal{L}}(x,y)$ is finite.
\end{theorem}
\begin{proof}[\footnotesize \textbf{Proof}]
	Since $\mathcal{L}$ is a translation-invariant operator, we may, without loss of generality, assume that a given test function satisfies
	$
	f(x)=0
	$ and $f(y)=\dist(x,y)$. The function
$
	\nicefrac{1}{2}\left(\dist(x,z)-f(z)\right)
$
	is a non-negative $1$-Lipschitz function that attains a global minimum at $x$. Therefore, by the comparison principle with range zero, we deduce
	\begin{equation*}
	\mathcal{L}\left(\dist(x,\cdot)\right)(x)-C_1(x)
	\le
	\mathcal{L}f(x)
	\le
	\mathcal{L}\left(\dist(x,\cdot)\right)(x)+C_1(x).
	\end{equation*}
	Similarly,
	\begin{equation*}
	\nicefrac{1}{2}\left(\dist(y,z)-f(z)+\dist(x,y)\right)
	\end{equation*}
	is non-negative and has a global minimum at $y$. Hence, by the hypotheses, we also obtain
	\begin{equation*}
	\mathcal{L}\left(\dist(x,\cdot)\right)(x)-C_1(y)
	\le
	\mathcal{L}f(y)
	\le
	\mathcal{L}\left(\dist(y,\cdot)\right)(y)+C_1(y).
	\end{equation*}
	Therefore, $\nabla_{\vv{yx}}\mathcal{L}f$ is bounded, and the infimum is finite.
\end{proof}
\begin{theorem}
	$\mathcal Ric_{\mathcal{L}}(x,y)$ is finite for all semi-local at controlled scale and translation-invariant operators $\mathcal{L}$ that are comparable to multiplication. 
\end{theorem}
\begin{proof}[\footnotesize \textbf{Proof}]
	We may assume that our test functions satisfy 	
	$
	f(x)=0
	$ and $f(y)=\dist(x,y)$. Consider the associated $1$-Lipschitz function $\widehat{f}$ as in \hyperref[lem:Lip-taper]{Lemma~\ref{lem:Lip-taper}}, with
	$
	\Uplambda>R(x)
	$.
	Thus,
$
	\nabla_{\vv{yx}}\mathcal{L}\widehat{f}
	=
	\nabla_{\vv{yx}}\mathcal{L}f
$.
	This argument shows that
	\begin{equation*}
	\inf_{\substack{
			f\in\mathcal{C}_{\sf fs}(\G)\\
			f\in\Lip(1)\\
			\nabla_{\vv{xy}}f=1}}
	\nabla_{\vv{yx}}\mathcal{L}f
	=
	\inf_{\substack{
			h\in\mathcal{C}_{\sf fs}(\G)\\
			h\in\Lip(1)\\
			\supp(h)\subset B_{2\zeta}(x)}}
	\nabla_{\vv{yx}}\mathcal{L}h,
	\end{equation*}
where
$
\zeta = \Uplambda + s(\Uplambda)
$.
Now, since any such $h$ satisfies $|h|\le 2\zeta$, the claim follows from the fact that $\mathcal{L}$ is comparable to multiplication (see \hyperref[defn:rough-comp-mult]{Definition~\ref{defn:rough-comp-mult}}).
\end{proof}
\begin{remark}
\textsl{Clearly, for two operators $\mathcal{L}_1$ and $\mathcal{L}_2$ as in the above two theorems, the curvature concavity 
	\begin{equation*}
		\Ric_{t\mathcal{L}_1+(1-t)\mathcal{L}_2}
		\ge
		t\Ric_{\mathcal{L}_1}
		+
		(1-t)\Ric_{\mathcal{L}_2}.
	\end{equation*}
holds and can be viewed as the infinitesimal (first-order) version of the concavity of $\Ric_\eps$.}
\end{remark}
	\renewcommand{\baselinestretch}{1.2}
\bibliography{OlRic-semgrp}
\bibliographystyle{amsalpha}
\vspace{20pt}
\end{document}